\documentclass[11pt,a4paper,final]{amsart}

\usepackage{amsmath}
\usepackage{paralist}
\usepackage{graphics}
\usepackage{graphicx}
\usepackage{epstopdf}
\usepackage{amssymb}
\usepackage{amsthm}
\usepackage{color}
\usepackage{multirow}
\usepackage{booktabs}
\usepackage{bm}
\usepackage{mathrsfs}

\usepackage[normalem]{ulem}

\newtheorem{theorem}{Theorem}[section]

\newtheorem{proposition}{Proposition}

\newtheorem{remark}{Remark}

\newcommand{\mc}[1]{{\mathcal #1}}
\newcommand{\mb}[1]{{\mathbf #1}}

\DeclareMathOperator*\uplim{\overline{lim}}

\begin{document}

\title[Motion of a family of interacting curves in {space}]{Qualitative and numerical aspects of a motion of a family of interacting curves in {space}}

\author{Michal Bene\v{s}${}^{1}$}
\author{Miroslav Kol\'a\v{r}${}^{1}$}
\address{${}^{1}$ Department of Mathematics, Faculty of Nuclear Sciences and Physical Engineering Czech Technical University in Prague, Trojanova 13, Prague, 12000, Czech Republic}
\author{Daniel \v{S}ev\v{c}ovi\v{c}${}^{2}$}
\address{${}^{2}$ Department of Applied Mathematics and Statistics, Faculty of Mathematics Physics and Informatics, Comenius University, Mlynsk\'a dolina, 842 48, Bratislava, Slovakia. Corresponding author: {\tt sevcovic@fmph.uniba.sk} }

\begin{abstract}
In this  article  we investigate a system of geometric evolution equations describing a curvature driven motion of a family of 3D curves in the normal and binormal directions. Evolving curves may be subject of mutual interactions having both local or nonlocal character where the entire curve may influence evolution of other curves.  Such an evolution and interaction can be found in applications. We explore the direct Lagrangian approach for treating the geometric flow of such interacting curves. Using the abstract theory of nonlinear analytic semi-flows, we are able to prove local existence, uniqueness and continuation of classical H\"older smooth solutions to the governing system of nonlinear parabolic equations. Using the finite volume method, we construct an efficient numerical scheme solving the governing system of nonlinear parabolic equations. Additionally, a nontrivial tangential velocity is considered allowing for redistribution of discretization nodes. We also present several computational studies of the flow combining the normal and binormal velocity and considering nonlocal interactions.



\end{abstract}

\maketitle

\section{Introduction}

In this article we investigate motion of a family $\{\Gamma^i_t, t\ge 0, i=1,\ldots, n\}$ of interacting curves evolving in three dimensional Euclidean space (3D) according to the geometric evolution law:
\begin{equation}
\partial_t\mb{X}^i = v^i_N \mb{N^i} + v^i_B \mb{B^i}  +  v^i_T \mb{T}^i, \quad i=1, \ldots, n,
\label{eq:general}
\end{equation}
where the unit tangent $\mb{T}^i$, normal $\mb{N^i}$ and binormal $\mb{B^i}$ vectors form the Frenet frame. We explore the direct Lagrangian approach to treat the geometric motion law (\ref{eq:general}). The evolving curves $\Gamma^i_t$ are parametrized as $\Gamma^i_t = \{ \mb{X}^i(u,t), u\in I, t\ge 0\}$ where  $\mb{X}^i:I\times[0,\infty) \to \mathbb{R}^3$ is a smooth mapping. Hereafter,  $I=\mathbb{R}/\mathbb{Z}\simeq S^1$ denotes the periodic interval $I=[0,1]$ isomorphic to the unit circle $S^1$ with $\partial I = \emptyset$. We assume the scalar velocities $v^i_N, v^i_T, v^i_B$ to be smooth functions of the position vector $\mb{X}^i\in\mathbb{R}^3$, the curvature $\kappa^i$, the torsion $\tau^i$, and of all parametrized curves $\Gamma^j, j=1,\ldots,n$, i.e.
\[
v^i_K= v^i_K(\mb{X}^i, \kappa^i, \tau^i, \mb{T}^i, \mb{N}^i, \mb{B}^i, \Gamma^1, \ldots, \Gamma^n), \quad K\in\{T, N, B\}, \quad i=1, \ldots, n.
\]

Motion (\ref{eq:general}) of one-dimensional structures forming space curves can be  identified in variety of problems arising in science and engineering. Among them, one of the oldest is the dynamics of vortex structures formed along a one-dimensional curve, frequently a closed one, forming a vortex ring. The investigation of these structures dates back to Helmholtz \cite{Helmholtz1858}. Since then, the importance of vortex structures for both understanding nature and improving aerospace technology is reflected in many publications, from which Thomson \cite{Thomson:67}, Da Rios \cite{Rios:06}, Betchov \cite{Betchov:65}, Arms and Hama \cite{Arms:65} or Bewley \cite{Bewley:08} are a sample only. Vortex structures can be relatively stable in time and may contribute to weather behavior, e.g. tornados, or accompany volcanic activity (c.f. Fukumoto \emph{et al.} \cite{Fukumoto1987, Fukumoto1991}, Hoz and Vega \cite{Hoz2014}, Vega \cite{Vega2015}). Particular vortex linear structures can interact each with other and exhibit interesting  dynamics, e.g. known as frog leaps (c.f. Mariani and Kontis \cite{Mariani2010}). A comprehensive review of research of vortex rings can be found in Meleshko \emph{et al.} \cite{Meleshko:12}.

One-dimensional structures can also be formed within the crystalline lattice of solid materials. As described, e.g. by Mura \cite{Mura}, some defects  of the crystalline lattice (voids or interstitial atoms) can be organized along planar curves in glide planes. These structures are called the dislocations and are responsible for macroscopic material properties  explored in the everyday engineering practice (see Hirth and Lothe \cite{Hirth} or Kubin \cite{Kubin}). The dislocations can move along the glide planes, be influenced by the external stress field in the material as well as by the force field of other dislocations. Such interaction can lead to the change of the glide plane (cross-slip) where  the motion becomes three-dimensional (see Devincre \emph{et al.} \cite{Devincre} or Pau{\v s}  \emph{et al.} \cite{Paus2013} or Kol{\'a}{\v r} \emph{et al.} \cite{PBKK:21}).

Certain class of nano-materials is produced by electrospinning - jetting polymer solutions in high electric fields into ultrafine nanofibers (see Reneker \cite{Reneker}, Yarin \emph{et al.} \cite{Yarin}, He \emph{et al.} \cite{He}). These structures move freely in space according to (\ref{eq:general}) before they are collected to form the material with desired features. The motion of nano-fibers as open curves in 3D is a combination of curvature and elastic response to the external electric forces (see Xu \emph{et al.} \cite{Xu}). As the nano-fibers are produced from a solution, they are subject of a drying process during electrospinning and may be considered as 3D objects with internal mass transfer, in detailed models (see \cite{Wu}).

Some linear molecular structures with specific properties exist inside cells and exhibit specific dynamics in terms of (\ref{eq:general}) in space, which is rather a result of chemical reactions. They can interact with other structures as described in Fierling \emph{et al.} in \cite{Fierling2016} where the deformations and twist of fluid membranes by adhering stiff amphiphilic filaments have been studied, or in Shlomovitz \emph{et al.} \cite{Shlomovitz:09},  Shlomovitz \emph{et al.} \cite{Shlomovitz:11}, Roux \emph{et al.} \cite{Roux:06}, Kang \emph{et al.} \cite{Kang:17} or in Glagolev \emph{et al.} \cite{Glagolev:18}.

The motion of curves in space or along manifolds has also been explored, e.g. in optimization of the truss construction and architectonic design (see Reme\v{s}\'{\i}kov\'a \emph{et al.} \cite{MS2014}), in the virtual colonoscopy \cite{MU2014}, in the numerical modeling of the wildland-fire propagation (see Ambro\v z \emph{et al.} \cite{ambroz2019}), or in the satellite-image segmentation (in Mikula \emph{et al.} \cite{M2021}).

Theoretical analysis of the motion of space curves is contained, among first, in papers by Altschuler and Grayson in \cite{altschuler1991} and \cite{altschuler1992}. The motion of space curves became useful tool in studying the singularities of the two-dimensional curve dynamics.  Nonlocal curvature driven flows, especially in case of planar curves, have been studied e.g. by Gage and Epstein  \cite{Gage86}, \cite{EpsteinGage}. Nonlocal curvature flows were treated by the Cahn-Hilliard theory in \cite{RuSte92} and in \cite{BroSto97}. Conserved planar curvature flow has been further investigated by Bene{\v s}, Kol{\' a}{\v r},  and \v{S}ev\v{c}ovi\v{c} in \cite{KoBeSe:14,matcom,BKS2017}. Recently, Bene{\v s}, Kol{\' a}{\v r},  and \v{S}ev\v{c}ovi\v{c} analyzed the flow of  planar curves with mutual interactions  in \cite{BKS2020}.

Recent theoretical results in the analysis of vortex filaments are provided by Jerrard and Seis \cite{Jerrard:17}. The dynamics of curves driven by curvature in the binormal direction is discussed by Jerrard and Smets in \cite{Jerrard:15}. Particular issues were numerically studied by Ishiwata and Kumazaki in \cite{Ishiwata:12}.

Curvature driven flow in a higher dimensional Euclidean space and comparison to the motion of hypersurfaces with the constrained normal velocity have been studied by  Barrett \emph{et al.} \cite{Barret2010, Barret2012}, Elliott and Fritz \cite{Elliot2017}, Minar\v{c}\'{\i}k, Kimura  and  Bene{\v s} in \cite{MB2019}. Gradient-flow approach is explored by Laux and Yip \cite{Laux:19}. Long-term behavior of the length shortening flow of curves in $\mathbb{R}^3$ has been analyzed by  Minar\v{c}\'{\i}k and  Bene{\v s} in \cite{MB2020}.

More specifically, we focus on the analysis of the motion of a family of curves evolving in 3D and satisfying the law
\begin{equation}
\partial_t\mb{X}^i = a^i \partial^2_{s^i} \mb{X}^i + b^i (\partial_{s^i}\mb{X}^i\times\partial^2_{s^i} \mb{X}^i)  + \mb{F}^i,
\quad i=1, \ldots, n,
\label{eq:ab}
\end{equation}
where $a^i=a^i(\mb{X}^i, \mb{T}^i) \ge 0$, and $b^i=b^i(\mb{X}^i, \mb{T}^i)$ are bounded and smooth functions of their arguments, $\mb{T}^i$ is the unit tangent vector to the curve and $s^i$ is the unit arc-length parametrization of the curve $\Gamma^i$ (see Section 2). The source forcing term $\mb{F}^i$ is assumed to be a smooth and bounded function. It may depend on the position and tangent vectors of the $i$-th curve and integrals over other interacting curves as follows:
\begin{equation}
\mb{F}^i = \mb{F}^i(\mb{X}^i, \mb{T}^i, \gamma^{i1}, \ldots, \gamma^{in}) \quad \text{where}\ \
\gamma^{ij}(\mb{X}^i, \Gamma^j) = \int_{\Gamma^j} f^{ij}(\mb{X}^i, \mb{T}^i, \mb{X}^j, \mb{T}^j) ds^j,
\label{eq:F}
\end{equation}
and $f^{ij}:\mathbb{R}^3\times \mathbb{R}^3 \times \mathbb{R}^3\times \mathbb{R}^3\to \mathbb{R}^3, i,j=1,\ldots, n$,  are given smooth functions.
Since $\partial^2_s \mb{X}^i =\kappa^i \mb{N}^i$ and $\mb{B}^i = \mb{T}^i \times \mb{N}^i$ (see Section~\ref{sec:equations}) the relationship between geometric equations (\ref{eq:general}) and (\ref{eq:ab}) reads as follows:
\begin{equation}
v^i_N = a^i \kappa^i + \mb{F}^i \cdot \mb{N}^i, \quad
v^i_B = b^i \kappa^i + \mb{F}^i \cdot \mb{B}^i, \quad
v^i_T = \mb{F}^i \cdot \mb{T}^i.
\label{eq:rel}
\end{equation}
The system of equations (\ref{eq:ab}) is subject to initial conditions
\begin{equation}
\mb{X}^i(u,0)  = \mb{X^i_0}(u), u\in I, \quad i=1, \ldots, n,
\label{init-ab}
\end{equation}
representing parametrization of the family of initial curves $\Gamma^i_0, i=1, \ldots, n$.

As an example of nonlocal source terms $\mb{F}^i, i=1,\ldots,n,$ we can consider a flow of $n=2$ interacting curves evolving in 3D according to the geometric equations:
\begin{equation}
\label{biot-savart}
\begin{split}
\partial_t\mb{X}^1 &= \partial_s\mb{X}^1\times\partial^2_s \mb{X}^1 + \gamma^{12}(\mb{X}^1, \Gamma^2),
\\
\partial_t\mb{X}^2 &= \partial_s\mb{X}^2\times\partial^2_s \mb{X}^2 + \gamma^{21}(\mb{X}^2, \Gamma^1),
\end{split}
\end{equation}
where the nonlocal source term has the form:
\begin{equation}
\gamma^{ij}(\mb{X}^i, \Gamma^j) = \int_{\Gamma^j}
\frac{(\mb{X}^i-\mb{X}^j)\times \mb{T}^j}{|\mb{X}^i-\mb{X}^j|^3} ds^j .
\label{biot-savart-force}
\end{equation}
It represents the Biot-Savart law measuring the integrated influence of points $\mb{X}^j$ belonging to the second curve $\Gamma^j=\{ \mb{X}^j(u), u\in[0,1]\}$ at a given point $\mb{X}^i$ belonging to the first interacting curve $\Gamma^i$. In this example $a^i=0$ and $b^i=1$. Such a flow is analyzed in a more detail in Subsection~\ref{binormalonly}. In the case of a special configuration of the initial curves the dynamics can be reduced to a solution to a system on nonlinear ODEs. On the other hand, if $a^i>0$ and $b^i\in\mathbb{R}$ there are no explicit or semi-explicit solutions, in general. Therefore a stable numerical discretization scheme has to be developed. The scheme involving a nontrivial tangential velocity is derived and presented in Subsection~\ref{normalbinormal}. For such a configuration of normal $a^i>0$ and binormal $b^i$ components of the velocity we establish local existence, uniqueness and continuation of classical H\"older smooth solutions in Section 4. Here, we generalize methodology and technique of proofs of local existence, uniqueness and continuation provided in \cite{BKS2020} to the case of combined motion of closed space curves in normal and binormal direction with mutual nonlocal interactions. The novelty and main contribution of this part is the result on existence and uniqueness of classical solutions for a system on $n$ evolving curves in $\mathbb{R}^3$ with mutual nonlocal interactions including, in particular, the vortex dynamics evolved in the normal and binormal directions and external force of the Biot-Savart type, or evolution of interacting dislocation loops.

To avoid singularities in (\ref{biot-savart-force}) arising in intersections of $\Gamma^i$ and $\Gamma^j$ one can regularize the expression for $\gamma^{ij}$ as follows
\begin{equation}
\gamma^{ij}_\delta(\mb{X}^i, \Gamma^j) = \int_{\Gamma^j}
\frac{(\mb{X}^i-\mb{X}^j)\times \mb{T}^j}{(\delta^2+ |\mb{X}^i-\mb{X}^j|^2)^{3/2}} ds^j,
\label{biot-savart-force-regularized}
\end{equation}
where $\delta>0$ is a small regularization parameter.

In general, the flow of $n\ge 2$ interacting curves involving the Biot-Savart law is governed by the system of $n$ evolutionary equations:
\begin{equation}
\partial_t\mb{X}^i = \partial_s\mb{X}^i\times\partial^2_s \mb{X}^i
+ \sum_{j\not= i}\gamma^{ij}(\mb{X}^i, \Gamma^j), \quad i=1,\ldots, n.
\label{biot-savart-general}
\end{equation}

The paper is organized as follows. In the next Section, we recall principles of the direct Lagrangian approach for solving normal and binormal curvature driven flows of a family of interacting plane curves in 3D. In Section 2 we derive a system of nonlocal evolution partial differential equations for parametrizations of a family of evolving curves. Section 3 is focused on the role of a tangential velocity. We will show that a suitable choice of tangential velocity leads to construction of an efficient and stable numerical scheme for solving the governing system of nonlinear parabolic equations in Section 5. Secondly, it helps to simplify the proof of local existence of classical solutions (see Section 4). Local existence, uniqueness, and continuation of classical H\"older smooth solutions is shown in Section~4. The method of the proof is based on the abstract theory of analytic semi-flows in Banach spaces due to Angenent \cite{Angenent1990, Angenent1990b}. A numerical discretization scheme is derived in Section~5. We apply the flowing finite volume method for discretization of spatial derivatives and the method of lines for solving the resulting system of ODEs. Finally, examples of evolution of interacting curves are presented in Section~6. Interactions are modeled by means of the Biot-Savart nonlocal law. We show examples of interacting curves following the motion with binormal velocity only as well as evolution of arbitrary curves evolving in both normal and binormal directions.

\section{Dynamic governing equations for geometric quantities}
\label{sec:equations}

Assume the family of evolving curves is parametrized as follows: $\Gamma^i_t = \{ \mb{X}^i(u,t), u\in I, t\ge 0\}$ where  $\mb{X}^i:I\times[0,\infty)\to \mathbb{R}^3$ is a smooth mapping.  For brevity we drop the superscript $i$ and we let $\mb{X}=\mb{X}^i$ wherever it is not necessary. Then the unit arc-length parametrization $s$ is given by  $ds = |\partial_u\mb{X}| du$. The unit tangent vector is given by $\mb{T} = \partial_s \mb{X}$. 
In the case when the curvature $\kappa = |\mb{T} \times \partial_s\mb{T}|>0$ is strictly positive, we can define the so-called Frenet frame. It means that the unit normal and binormal vectors $\mb{N}$ and $\mb{B}$ can be uniquely defined as follows: $\mb{N} = \kappa^{-1} \partial_s \mb{T}$, $\mb{B} = \mb{T} \times \mb{N}$. These unit vectors satisfy the following identities:
\[
\mb{B} = \mb{T} \times \mb{N},
\qquad
\mb{T} = \mb{N} \times \mb{B},
\qquad
\mb{N} = \mb{B} \times \mb{T},
\]
and the Frenet-Serret formulae:
\[
\frac{d}{ds}
\left(
\begin{array}{c}
     \mb{T}
     \\
     \mb{N}
     \\
     \mb{B}
\end{array}
\right)
=
\left(
\begin{array}{ccc}
     0 & \kappa & 0
     \\
     -\kappa & 0 &\tau
     \\
     0 & -\tau & 0
\end{array}
\right)
\left(
\begin{array}{c}
     \mb{T}
     \\
     \mb{N}
     \\
     \mb{B}
\end{array}
\right),
\]
where $\tau$ is the torsion of a curve. For $\kappa>0$ the torsion $\tau$ is given by
\[
\tau =
\kappa^{-2} (\mb{T}\times\partial_s\mb{T}) \cdot \partial_s^2\mb{T}
= \kappa^{-2} (\partial_s\mb{X}\times\partial_s^2\mb{X}) \cdot \partial_s^3\mb{X} .
\]
Indeed, as $\partial_s\mb{B} = \partial_s\mb{T}\times \mb{N} + \mb{T}\times\partial_s\mb{N}
= \mb{T}\times\partial_s(\kappa^{-1}\partial_s\mb{T})  = \kappa^{-1}(\mb{T}\times\partial_s^2\mb{T})$, we obtain
\[
\tau = -\partial_s\mb{B} \cdot \mb{N} = - \kappa^{-1}(\mb{T}\times\partial_s^2\mb{T})\cdot \kappa^{-1}\partial_s\mb{T} =
 \kappa^{-2} (\mb{T}\times\partial_s\mb{T}) \cdot \partial_s^2\mb{T} .
\]

Concerning the dynamical governing equations we have the following proposition. Some of these identities have been already discovered as a particular case by other authors (see e.g. \cite{MB2020}, \cite{MB2019}). Our aim is to provide evolution equations  general settings of normal $v_N$, binormal $v_B$, and tangent velocities $v_T$. Although our approach is based on the analysis and numerical solution of the position vector equation (\ref{eq:ab}), we provide the dynamic equations for the curvature and torsion in the following proposition. 

\begin{proposition}\label{prop-1}
Assume a family of curves $\Gamma_t, t\ge 0,$ is evolving in 3D according to the geometric law:
\[
\partial_t\mb{X} = v_N \mb{N} + v_B \mb{B}  +  v_T \mb{T}.
\]
Then the unit vectors $\mb{N}, \mb{B}, \mb{T}$ forming the Frenet frame satisfy the following system of evolution partial differential equations:
\begin{eqnarray*}
\partial_t \mb{T} &=&
\left( \partial_s v_N  + \kappa v_T - \tau v_B\right)\mb{N}
+ \left( \partial_s v_B +  \tau v_N\right)\mb{B} ,
\\
\kappa \partial_t \mb{N} 
&=&
-\kappa \left( \partial_s v_N  + \kappa v_T - \tau v_B\right)\mb{T}
+ \left( \partial_s^2 v_B +  \partial_s (\tau v_N)
+\tau \left( \partial_s v_N  + \kappa v_T - \tau v_B\right)
\right)\mb{B},
\\
\kappa \partial_t\mb{B} 
&=& 
- \kappa \left( \partial_s v_B +  \tau v_N\right)\mb{T}
- \left( \partial_s^2 v_B +  \partial_s (\tau v_N)
+\tau \left( \partial_s v_N  + \kappa v_T - \tau v_B\right)
\right)\mb{N} .
\end{eqnarray*}
The local length element $g=|\partial_u\mb{X}|$ and the commutator $[\partial_t, \partial_s]:=\partial_t\partial_s - \partial_s\partial_t$ satisfy
\[
\partial_t g  =  (-\kappa v_N + \partial_s v_T ) g, \quad
\partial_t ds = (-\kappa v_N + \partial_s v_T ) ds, \quad
\partial_t\partial_s - \partial_s\partial_t = (\kappa v_N - \partial_s v_T )\partial_s .
\]
The curvature $\kappa$ and torsion $\tau$ 
(for $\kappa(s,t)>0$) 
satisfy the evolution equations:
\begin{eqnarray*}
\partial_t \kappa &=&  \partial_s^2 v_N  +  \kappa^2 v_N + v_T \partial_s \kappa
 - \partial_s(\tau v_B)
-\tau\partial_s v_B -  \tau^2 v_N ,
\\
\partial_t\tau &=&
\kappa\left( \partial_s v_B +  \tau v_N\right)
+
\partial_s\left(
\kappa^{-1}\left( \partial_s^2 v_B +  \partial_s (\tau v_N)
+\tau \left( \partial_s v_N  + \kappa v_T - \tau v_B\right)
\right)  \right)
\nonumber
\\
&& +  \tau(\kappa v_N-\partial_s v_T) .
\end{eqnarray*}
\end{proposition}

\begin{proof}
Denote $g=|\partial_u X|$. Then $ds = g du$. Using Frenet-Serret formulae we have
\begin{eqnarray}
\partial_t \mb{T} &=& \partial_t(g^{-1}\partial_u \mb{X}) = - g^{-1}\partial_t g\, \mb{T} + \partial_s\partial_t\mb{X}
= - g^{-1}\partial_t g \mb{T} + \partial_s (v_N \mb{N} + v_T \mb{T} + v_B \mb{B} )
\nonumber\\
&=&
\left( - g^{-1}\partial_t g + \partial_s v_T -\kappa v_N\right)\mb{T}
+
\left( \partial_s v_N  + \kappa v_T - \tau v_B\right)\mb{N}
+ \left( \partial_s v_B +  \tau v_N\right)\mb{B} .
\nonumber
\end{eqnarray}
Since $0=\partial_t (\mb{T}\cdot \mb{T})= 2 (\mb{T}\cdot \partial_t\mb{T})$ we have
\[
\partial_t \mb{T} =
\left( \partial_s v_N  + \kappa v_T - \tau v_B\right)\mb{N}
+ \left( \partial_s v_B +  \tau v_N\right)\mb{B},
\]
and, as a consequence, $\partial_t g  =  (-\kappa v_N + \partial_s v_T ) g$, and
$\partial_t\partial_s = \partial_s\partial_t + (\kappa v_N - \partial_s v_T )\partial_s$ because $\partial_t\partial_s  = \partial_t(g^{-1}\partial_u)=g^{-1} \partial_u\partial_t  -g^{-2} \partial_t g \partial_u$. Next
\begin{eqnarray}
\kappa \partial_t \mb{N} &=& \kappa \partial_t(\kappa^{-1}\partial_s \mb{T}) = - \partial_t \kappa\, \mb{N}
+
\partial_s\partial_t\mb{T}
+ (\kappa v_N - \partial_s v_T )\partial_s\mb{T}
\nonumber\\
&=&
\left( -\partial_t \kappa + \kappa^2 v_N - \kappa \partial_s v_T
\right)\mb{N}
+ \partial_s\partial_t\mb{T}
\nonumber
\\
&=&
\left( -\partial_t \kappa + \kappa^2 v_N + v_T \partial_s \kappa
+ \partial_s^2 v_N  - \partial_s(\tau v_B)
\right)\mb{N}
\nonumber
\\
&& + \left( \partial_s v_N  + \kappa v_T - \tau v_B\right)\partial_s\mb{N}
 + \left( \partial_s v_B +  \tau v_n\right)\partial_s\mb{B}
+ \left( \partial_s^2 v_B +  \partial_s (\tau v_N)\right)\mb{B}
\nonumber
\\
&=&
\left( -\partial_t \kappa + \kappa^2 v_N + v_T \partial_s \kappa
+ \partial_s^2 v_N  - \partial_s(\tau v_B)
-\tau \left( \partial_s v_B +  \tau v_N\right) \right)\mb{N}
\nonumber
\\
&& -\kappa \left( \partial_s v_N  + \kappa v_T - \tau v_B\right)\mb{T}
+ \left( \partial_s^2 v_B +  \partial_s (\tau v_N)
+\tau \left( \partial_s v_N  + \kappa v_T - \tau v_B\right)
\right)\mb{B} .
\nonumber
\end{eqnarray}
Since $0=\partial_t (\mb{N}\cdot \mb{N})= 2 (\mb{N}\cdot \partial_t\mb{N})$ we have
\begin{eqnarray}
\kappa \partial_t \mb{N} 
&=&
-\kappa \left( \partial_s v_N  + \kappa v_T - \tau v_B\right)\mb{T}
+ \left( \partial_s^2 v_B +  \partial_s (\tau v_N)
+\tau \left( \partial_s v_N  + \kappa v_T - \tau v_B\right)
\right)\mb{B},
\nonumber
\end{eqnarray}
and, as a consequence,
\[
\partial_t \kappa =  \partial_s^2 v_N  +  \kappa^2 v_N + v_T \partial_s \kappa
 - \partial_s(\tau v_B)
-\tau\partial_s v_B -  \tau^2 v_N .
\]
Finally, as $\partial_t\mb{B} = \partial_t\mb{T}\times\mb{N} + \mb{T}\times\partial_t\mb{N}$ and $\mb{B}\times\mb{N}=-\mb{T}$ and $\mb{T}\times\mb{B}=-\mb{N}$ we have
\[
\kappa \partial_t\mb{B} = - \kappa \left( \partial_s v_B +  \tau v_N\right)\mb{T}
- \left( \partial_s^2 v_B +  \partial_s (\tau v_N)
+\tau \left( \partial_s v_N  + \kappa v_T - \tau v_B\right)
\right)\mb{N} .
\]
In the case when the curvature $\kappa(s,t)$ is strictly positive,
the evolution equation for the torsion $\tau$ can be deduced from the fact $\tau=-\partial_s\mb{B}\cdot\mb{N}$, i.e.
\begin{eqnarray}
\partial_t\tau &=& -\partial_t\partial_s\mb{B} \cdot \mb{N}
-\partial_s\mb{B} \cdot \partial_t\mb{N}
\nonumber\\
&=& -\left( \partial_s\partial_t\mb{B} +  (\kappa v_N-\partial_s v_T)\partial_s\mb{B}\right)\cdot\mb{N} + \tau \mb{N}\cdot\partial_t\mb{N}
\nonumber\\
&=& -\left( \partial_s\partial_t\mb{B}\right) \cdot\mb{N}  +  \tau(\kappa v_N-\partial_s v_T)
\nonumber\\
&=& -\partial_s\left(
- \left( \partial_s v_B +  \tau v_N\right)\mb{T}
- \kappa^{-1}\left( \partial_s^2 v_B +  \partial_s (\tau v_N)
+\tau \left( \partial_s v_N  + \kappa v_T - \tau v_B\right)
\right)\mb{N}
\right) \cdot\mb{N}
\nonumber\\
&& +  \tau(\kappa v_N-\partial_s v_T)
\nonumber
\\
&=&
\kappa\left( \partial_s v_B +  \tau v_N\right)
+
\partial_s\left(
\kappa^{-1}\left( \partial_s^2 v_B +  \partial_s (\tau v_N)
+\tau \left( \partial_s v_N  + \kappa v_T - \tau v_B\right)
\right)  \right)
\nonumber
\\
&& +  \tau(\kappa v_N-\partial_s v_T).
\nonumber
\end{eqnarray}

\end{proof}

As a consequence of the previous proposition we obtain the following results concerning temporal evolution of global quantities integrated over the evolving curves:

\begin{proposition}\label{prop-global}
Assume a family of curves $\Gamma_t, t\ge 0,$ evolving in 3D according to the geometric law:
\[
\partial_t\mb{X} = v_N \mb{N} + v_B \mb{B}  +  v_T \mb{T}.
\]
Then, the length $L(\Gamma)=\int_\Gamma ds$ and the generalized area $A(\Gamma)=\frac12 \int_\Gamma (\mb{X}\times \partial_s \mb{X})\cdot\mb{B}\, ds$ enclosed by $\Gamma$  satisfy the following identities:
\begin{eqnarray*}
\frac{d}{dt} L(\Gamma) &=& -\int_\Gamma \kappa v_N ds ,
\\
\frac{d}{dt} A(\Gamma) &=&
- \int_\Gamma  v_N ds
-\frac12 \int_\Gamma  (\mb{X}\times \partial_t \mb{X})\cdot\partial_s \mb{B} \, ds
+\frac12 \int_\Gamma (\mb{X}\times \partial_s  \mb{X})\cdot \partial_t\mb{B} \, ds .
\end{eqnarray*}
In particular, if the family  $\Gamma_t, t\ge 0,$  of curves evolves in parallel planes then $A(\Gamma)$ is the area enclosed by $\Gamma$, and $\frac{d}{dt} A(\Gamma) = -\int_\Gamma v_N ds $.

\end{proposition}

\begin{proof}
The first statement follows from the identity $\partial_t g  =  (-\kappa v_N + \partial_s v_T ) g$. Indeed,
\[
\frac{d}{dt} L(\Gamma) = \frac{d}{dt} \int_0^1 g du = \int_0^1 \partial_t g du = \int_\Gamma  (-\kappa v_N + \partial_s v_T) ds = -\int_\Gamma \kappa v_N ds ,
\]
because $\Gamma$ is a closed curve. Therefore $\int_\Gamma \partial_s v_T ds =0$.

As for the second statement, we have $A(\Gamma)=\frac12 \int_\Gamma (\mb{X}\times \partial_s \mb{X})\cdot\mb{B}\, ds = \frac12 \int_0^1 (\mb{X}\times \partial_u \mb{X})\cdot\mb{B}\, du$, and so
\begin{eqnarray*}
\frac{d}{dt} A(\Gamma) &=& \frac12 \int_0^1 (\partial_t \mb{X}\times \partial_u \mb{X})\cdot\mb{B}
+ (\mb{X}\times \partial_u \partial_t \mb{X})\cdot\mb{B}
+ (\mb{X}\times \partial_u  \mb{X})\cdot \partial_t\mb{B}
\, du
\\
 &=& \frac12 \int_\Gamma (\partial_t \mb{X}\times \partial_s \mb{X})\cdot\mb{B}
+ (\mb{X}\times \partial_s \partial_t \mb{X})\cdot\mb{B}
+ (\mb{X}\times \partial_s  \mb{X})\cdot \partial_t\mb{B} \, ds
\\
 &=& - \int_\Gamma  (\partial_s \mb{X}\times \partial_t \mb{X})\cdot\mb{B}\, ds
-\frac12 \int_\Gamma  (\mb{X}\times \partial_t \mb{X})\cdot\partial_s \mb{B} \, ds
+\frac12 \int_\Gamma (\mb{X}\times \partial_s  \mb{X})\cdot \partial_t\mb{B} \, ds
\\
 &=& - \int_\Gamma  v_N ds
-\frac12 \int_\Gamma  (\mb{X}\times \partial_t \mb{X})\cdot\partial_s \mb{B} \, ds
+\frac12 \int_\Gamma (\mb{X}\times \partial_s  \mb{X})\cdot \partial_t\mb{B} \, ds\, .
\end{eqnarray*}
In particular, if the  family of 3D curves $\Gamma_t, t\ge 0,$ evolves in parallel planes with the normal vector $\mb{b}$ then the binormal vector $\mb{B} = \pm \mb{b}/|\mb{b}|$ is a constant vector perpendicular to this plane. As a consequence,  $\partial_t\mb{B} = \partial_s\mb{B} =0$, and the proof of the last statement of the proposition follows from the fact that the enclosed area of a curve belonging to the plane $x_3=0$ is given by $A(\Gamma)=\frac12 \int_\Gamma x_1\partial_s x_2 - x_2\partial_s x_1 ds$, and $(Q\mb{a}\times Q\mb{d})\cdot Q\mb{c} = (\mb{a}\times \mb{d})\cdot \mb{c}$ for any rotation matrix $Q$ transforming the vector $\mb{b}$ to the vector $(0,0,1)^T$.
\end{proof}

\section{The role of tangential redistribution}
\label{redistribution}

The tangential velocity $v_T$ appearing in the geometric evolution (\ref{eq:ab}) has no impact on the shape of evolving family of curves $\Gamma^i_t, t\ge 0$. It means that the curves $\Gamma^i_t, t\ge 0,$ evolving according to the system of geometric equations:
\begin{equation}
\partial_t\mb{X}^i = a^i \partial^2_{s^i} \mb{X}^i + b^i (\partial_{s^i}\mb{X}^i\times\partial^2_{s^i} \mb{X}^i)  + \mb{F}^i +\alpha^i \mb{T}^i. \quad i=1,\ldots, n,
\label{eq:ab-alpha}
\end{equation}
do not depend on a particular choice of the total tangential velocity $v^i_T$ given by
\[
v^i_T = \mb{F}^i \cdot \mb{T}^i + \alpha^i .
\]
However, the tangential velocity has a significant impact on the analysis of evolution of curves from both the analytical as well as numerical points of view. It was shown by Hou et al. \cite{Hou}, Kimura \cite{Kimura}, Mikula and \v{S}ev\v{c}ovi\v{c} \cite{sevcovic2001evolution, MS2004, MMAS2004}, Yazaki and \v{S}ev\v{c}ovi\v{c} \cite{SevcovicYazaki2012}.  Barrett \emph{et al.} \cite{Barret2010, Barret2012}, Elliott and Fritz \cite{Elliot2017}, investigated the gradient and elastic flows for closed and open curves in $\mathbb{R}^d, d\ge 2$. They constructed a numerical approximation scheme using a suitable tangential redistribution. Kessler \emph{et al.} \cite{Kessler1984} and Strain \cite{Strain1989} illustrated the role of suitably chosen tangential velocity in numerical simulation of the two-dimensional snowflake growth and unstable solidification models. Later, Garcke \emph{et al.} \cite{Garcke2009} applied the uniform tangential redistribution in the theoretical proof of nonlinear stability of stationary solutions for curvature driven flow with triple junction in the plane. 

A suitable choice of $v_T$ can be very useful in order to prove local existence of solution. Furthermore, it can significantly help to construct a stable an efficient numerical scheme preventing from undesirable accumulation of grid points during curve evolution. Calculating the derivative ratio $g^i/L(\Gamma^i)$ with respect to time we obtain
\begin{equation}
\frac{\partial }{\partial t}
\frac{g^i}{L^i} = \frac{\partial_t g^i}{L^i} -\frac{g^i}{(L^i)^2} \frac{d L^i}{dt}  =\frac{g^i}{L^i}\left(
-\kappa^i v^i_N +\partial_{s^i} v^i_T +\frac{1}{L^i} \int_{\Gamma^i} \kappa v^i_N ds^i
\right),
\label{alpha}
\end{equation}
where $L^i=L(\Gamma^i_t)$. As a consequence, the relative local length $g^i/L^i$ is constant with respect to the time $t$, i.e.
\[
\frac{g^i(u,t)}{L(\Gamma^i_t)} =  \frac{g^i(u,0)}{L(\Gamma^i_0)}, \quad u\in I, t\ge 0,
\]
provided that the total tangential velocity $v^i_T$ satisfies:

\begin{equation}
\partial_{s^i} v^i_T= \kappa^i v^i_N  -\frac{1}{L^i} \int_{\Gamma^i} \kappa v^i_N ds^i,
\label{unifalpha}
\end{equation}
(c.f. Hou and Lowengrub \cite{Hou}, Kimura \cite{Kimura},  Mikula and \v{S}ev\v{c}ovi\v{c} \cite{sevcovic2001evolution}). Since $v^i_T=\mb{F}^i \cdot \mb{T}^i + \alpha^i$ the additional tangential velocity $\alpha^i$  given by
\begin{equation}
\label{eq:alpha}
\alpha^i(s^i)  = - \mb{F}^i(s^i) \cdot \mb{T}^i(s^i)  + \mb{F}^i(0) \cdot \mb{T}^i(0) + \alpha^i(0)
+ \int_0^{s^i} \kappa^i v^i_N ds^i  - s^i \frac{1}{L^i} \int_{\Gamma^i} \kappa^i v^i_N ds^i.
\end{equation}
$s^i\in [0, L^i]$, ensures that the relative local length $g^i/L^i$ is constant with respect to time, and
\[
g^i(u,t) =  g^i_0(u)\frac{L(\Gamma^i_t)}{L(\Gamma^i_0)}, \quad u\in I, t\ge 0, i=1,\ldots, n,
\]
where $g^i_0(u)= g^i(u,0)$. The tangential velocity is subject to the  normalization constraint $\int_{\Gamma^i} \alpha^i ds^i =0$.

Another suitable choice of the total tangential velocity $v^i_T$ is the so-called asymptotically uniform tangential velocity proposed and analyzed by Mikula and \v{S}ev\v{c}ovi\v{c} in \cite{MS2004, MMAS2004}. If
\begin{equation}
\partial_{s^i} v^i_T= \kappa^i v^i_N  -\frac{1}{L^i} \int_{\Gamma^i} \kappa v^i_N ds^i  + \left( \frac{L^i}{g^i} - 1\right) \omega ,
\label{alpha-asymptotic}
\end{equation}
then, using (\ref{alpha}) we obtain
\[
\lim_{t\to \infty} \frac{g^i(u,t)}{L(\Gamma^i_t)} =1
\]
uniformly with respect to $u\in[0,1]$ provided $\omega>0$. It means that the redistribution becomes asymptotically uniform. In the context of evolution of 3D curves or the curves evolving on a given surface, the concept uniform and asymptotically uniform redistribution has been analyzed and successfully implemented for various applications by Mikula and \v{S}ev\v{c}ovi\v{c} in \cite{MS2004, MS2014}, Mikula \emph{et al.} \cite{M2021}, Bene{\v s} \emph{et al.} \cite{0965-0393-24-3-035003}, Ambro\v{z} \emph{et al.} \cite{ambroz2019}, and others.

\begin{remark}\label{uniform}
Suppose that the initial curve $\Gamma_0$ is uniformly parametrized, i.e. $g_0(u)=|\partial_u \mb{X}(u,0)| = L(\Gamma_0)$. If $\alpha$ is a tangential velocity preserving the relative local length then
\[
g(u,t)=|\partial_u \mb{X}(u,t)| = L(\Gamma_t), \quad  \text{and}\ \ ds = L(\Gamma_t) du,\quad  s\in[0, L(\Gamma_t)].
\]
\end{remark}

\section{Existence and uniqueness of classical solutions}
\label{existence}

In this section we provide existence and uniqueness results for the system of nonlinear nonlocal equations (\ref{eq:ab-alpha}) governing the motion of interacting closed curves in 3D. The method of the proof of existence and uniqueness is based on the abstract theory of analytic semi-flows in Banach spaces due to DaPrato and Grisvard \cite{daprato}, Angenent \cite{Angenent1990, Angenent1990b}, Lunardi \cite{Lunardi1984}. 
Local existence and uniqueness of a classical H\"older smooth solution is based on analysis of the position vector equation (\ref{eq:ab-alpha}) in which we choose the uniform tangential velocity $\alpha^i$. It leads to a uniformly parabolic equation  (\ref{eq:ab-alpha}) provided the diffusion coefficients $a^i$ are uniformly bounded from below by a positive constant. As a consequence, assumptions on strict positivity of the curvature $\kappa^i$ and the existence of the Frenet frame are not required, in our method of the proof.
The main idea is to rewrite the system (\ref{eq:ab-alpha}) in the form of an initial value problem for the abstract parabolic equation:
\begin{equation}
\partial_t \mb{X} + \mathscr{F}(\mb{X}) = 0, \quad \mb{X}(0) = \mb{X}_0,
\label{abstratF}
\end{equation}
in a suitable Banach space. Furthermore, we have to show that, for any $\tilde{\mb{X}}$, the linearization $\mathscr{F}'(\tilde{\mb{X}} )$ generates an analytic semigroup and it belongs to the so-called maximal regularity class of linear operators mapping the Banach space $\mathcal{E}_1$ into Banach space $\mathcal{E}_0$.

Note that the principal part $a \partial^2_s \mb{X} + b (\partial_s\mb{X}\times\partial^2_s \mb{X})$ of the velocity vector $\partial_t\mb{X}$  can be expressed in the matrix form as follows:
\[
a \partial^2_s \mb{X} + b (\partial_s\mb{X}\times\partial^2_s \mb{X})
\equiv
{\mathcal A}(a,b,\partial_s\mb{X}) \partial^2_s \mb{X},
\]
where ${\mathcal A}(a,b,\mb{T})$ is a $3\times 3$ matrix,
\[
{\mathcal A}(a,b,\mb{T}) = a I + b [\mb{T}]_\times :=
\left(
\begin{array}{ccc}
     a & -b T_3 & b T_2
     \\
     b T_3 & a & -b T_1
     \\
     -b T_2 & b T_1 & a
\end{array}
\right).
\]
Clearly, the symmetric part $\frac12( {\mathcal A} + {\mathcal A}^T) = a I\succ 0 $ is a positive definite matrix for $a>0$. If $a=0$ then ${\mathcal A}$ is an indefinite and antisymmetric matrix, i.e., ${\mathcal A} = - {\mathcal A}^T$. For given values $a,b$ and a unit vector $\mb{T}$, the eigenvalues of the matrix ${\mathcal A}$ are: $\mu_1 = a, \mu_{2} = a - i b, \mu_3= a + i b$. It means that the governing equation:
\begin{equation}
\partial_t\mb{X} = {\mathcal A}(a,b,\partial_s\mb{X}) \partial^2_s \mb{X} + \mb{F}
+\alpha \mb{T}
\label{abstract-form}
\end{equation}
is of the parabolic type provided $a>0$ whereas it is of the hyperbolic type if $a=0$ and $b\not=0$. In the case of $n\ge 2$ interacting curves the system of governing equations reads as follows:
\begin{equation}
\begin{split}
\partial_t\mb{X}^1 &= {\mathcal A}(a^1,b^1,\partial_{s^1}\mb{X}^1) \partial^2_{s^1} \mb{X}^1
\ + \mb{F^1}(\mb{X}^1, \partial_{s^1}\mb{X}^1, \gamma^{11}, \ldots, \gamma^{1n}) \ +\alpha^1 \mb{T}^1,
\\
\vdots &
\\
\partial_t\mb{X}^n &= {\mathcal A}(a^n,b^n,\partial_{s^n}\mb{X}^n) \partial^2_{s^n} \mb{X}^n
+ \mb{F^n}(\mb{X}^n, \partial_{s^n}\mb{X}^n, \gamma^{n1}, \ldots, \gamma^{nn}) +\alpha^n \mb{T}^n,
\end{split}
\label{abstract-form-general}
\end{equation}
where $\gamma^{ij}=\gamma^{ij}(\mb{X}^i, \Gamma^j)$ for $i,j=1,\ldots, n$.

\subsection{Maximal regularity for parabolic equations with complex valued  diffusion functions}

Assume $0<\varepsilon<1$ and $k$ is a nonnegative integer. Let us denote by $h^{k+\varepsilon}(S^1)$ the so-called little H\"older space, i.e. the Banach space which is the closure of $C^\infty$ smooth functions in the norm  Banach space of $C^k$ smooth functions defined on the periodic domain $S^1$, and such that the $k$-th derivative is $\varepsilon$-H\"older smooth. The norm is being given as a sum of the $C^k$ norm and the H\"older semi-norm of the $k$-th derivative.

Among many important properties of H\"older spaces $h^{k+\varepsilon}(S^1)$ there is an interpolation inequality. Let $\varepsilon^{\prime\prime}, \varepsilon^{\prime}, \varepsilon \in (0,1), k^{\prime\prime}, k^{\prime}, k \in \mathbb{N}_0$ be such that  $ k^{\prime\prime} + \varepsilon^{\prime\prime} <  k^{\prime} + \varepsilon^{\prime} < k + \varepsilon$.
Then, for any $\delta>0$ there exists $C_\delta>0$ such that \begin{equation}
\Vert\varphi \Vert_{h^{k^{\prime}+\varepsilon^{\prime}}} \le \Vert\varphi \Vert_{h^{k+\varepsilon}}^\theta \Vert\varphi \Vert_{h^{k^{\prime\prime}+\varepsilon^{\prime\prime}}}^{1-\theta}
\le \delta \Vert\varphi \Vert_{h^{k+\varepsilon}} + C_\delta \Vert\varphi \Vert_{h^{k^{\prime\prime}+\varepsilon^{\prime\prime}}},
\label{interpolation}
\end{equation}
for any $\varphi\in h^{k+\varepsilon}(S^1)$, where $\theta=(k^{\prime}+ \varepsilon^{\prime} - k^{\prime\prime}  - \varepsilon^{\prime\prime})/ (k+\varepsilon - k^{\prime\prime}  - \varepsilon^{\prime\prime}) \in (0,1)$.

In what follows, we shall assume that the functions $a,b\in h^{1+\varepsilon}(S^1)$, and  $a>0$ is strictly positive. Let us define the following linear second order differential operators $A,B : h^{2+\varepsilon}(S^1) \to h^{\varepsilon}(S^1)$:
\begin{equation}
A\varphi =-\partial_u(a(\cdot)\partial_u\varphi), \qquad B\varphi =-\partial_u(b(\cdot)\partial_u\varphi), \quad \text{for}\ \varphi\in h^{2+\varepsilon}(S^1).
\label{AB-operators}
\end{equation}
The spectra $\sigma(A)\subset [0,\infty), \sigma(B)\subset\mathbb{R}$, consist of discrete real eigenvalues. Furthermore, the linear operators $\pm i B$ generate the $C^0$ group of linear operators $\{e^{\pm i B t}, t\in \mathbb{R}\}$. It means that the function $\xi(t) = e^{\pm i B t} \xi_0$ is a solution to the Schr\"odinger equation
\[
\partial_t\xi = \pm i B \xi, \quad \xi(0)=\xi_0.
\]
Recall that the spectrum $\sigma(B)$ consists of real eigenvalues. Hence the linear operator $e^{\pm i B t}$ is bounded in the space $L(C^k(S^1))$ uniformly with respect to $t\ge 0$. Since $h^{k+\varepsilon}(S^1)$ is an interpolation space between $C^k(S^1)$ and $C^{k+1}(S^1)$ there exists a constant $c_0>0$ depending on the function $b$ only and such that
\begin{equation}
\Vert e^{\pm i B t} \Vert_{L(h^{k+\varepsilon}(S^1))} \le c_0 \quad \text{for}\ k=0,2\ \text{and any}\ t\ge 0.
\label{Bbound}
\end{equation}
Moreover, $\lim_{t\to 0} e^{\pm i B t} = I$ in the respective norms of linear operators, $k=0,2$.

Next, we shall prove the maximal regularity of solutions to the linear evolutionary equation:
\begin{equation}
\partial_t \varphi + ( A+ i B)\varphi = f, \quad t\ge 0, \qquad \varphi(0)=\varphi_0.
\label{ABequation}
\end{equation}
That is to show the existence of a unique solution $\varphi\in{\mathcal H}_1(0,T)$ for the given right-hand side $f\in{\mathcal H}_0(0,T)$ and initial condition $\varphi_0\in h^{2+\varepsilon}(S^1)$ and $T>0$  Here we have denoted by ${\mathcal H}_0, {\mathcal H}_1$ the following Banach spaces:
\begin{equation}
{\mathcal H}_1(0,T) = C([0,T], h^{2+\varepsilon}(S^1)) \cap C^1([0,T], h^{\varepsilon}(S^1)),
\quad
{\mathcal H}_0(0,T) = C([0,T], h^{\varepsilon}(S^1)).
\end{equation}

Consider the transformed function $\psi=e^{i B t} (\varphi-\varphi_0)$. Then $\varphi$ is a solution to (\ref{ABequation}) if and only if $\psi$ is a solution to the equation:
\begin{equation}
\partial_t \psi +  A \psi = R_t \psi + \hat f, \quad t\ge 0, \qquad \psi(0)=0,
\label{Aequation}
\end{equation}
where $R_t = A - e^{i B t} A e^{-i B t}$, $\hat f = e^{i B t} (f  - ( A+ i B)\varphi_0)$.  Clearly, $\hat f\in {\mathcal H}_0(0,T)$. Recall that the linear operator $A=- \partial_u(a\partial_u)$ generates an analytic semi-group of operators $\{e^{-A t}, t\ge 0\}$. Moreover, it belongs to the so-called maximal regularity class ${\mathcal M}(h^{2+\varepsilon}, h^{\varepsilon})$ (c.f. \cite{Angenent1990b}, \cite{Angenent1990}, \cite{daprato}). It means that the linear operator $\partial_t + A : {\mathcal H}_1(0,T) \to {\mathcal H}_0(0,T)$ is invertible, i.e. for any $\hat g: \in {\mathcal H}_0(0,T)$ and $\psi_0\in h^{2+\varepsilon}(S^1)$ there exists a unique solution $\psi\in {\mathcal H}_1(0,T)$ of the initial value problem $\partial_t\psi  + A\psi = \hat g, \ \psi(0)=\psi_0$, and,  $\Vert \psi\Vert_{{\mathcal H}_1(0,T)} \le c_1 ( \Vert \hat g\Vert_{{\mathcal H}_0(0,T)} + \Vert\psi_0 \Vert_{h^{2+\varepsilon}} ) $ where  $c_1>0$ is a constant.

Since $\lim_{t\to 0} R_t = 0$ there exists a time $0<T_0\le T$ depending on the functions $a$ and $b$ only, and such that $\Vert (\partial_t +A)^{-1} R_t \Vert_{L({\mathcal H}_1(0,T_0))} < 1$. As a consequence, the operator $I - (\partial_t +A)^{-1} R_t$ is invertible in the space ${\mathcal H}_1(0,T_0)$. That is the operator $\partial_t  + (A + i B)$ is invertible on the time interval $[0, T_0]$. Now, starting from the initial condition $\psi_0=\psi(T_0)$ we can continue the solution $\psi$ over the larger interval $[0, T_0]\cup[T_0, 2T_0]$. Continuing in this manner, we can conclude that the operator $A + i B$ generates an analytic semigroup $e^{-(A + i B) t}, t\ge 0,$ and it belongs to the maximal regularity class ${\mathcal M}(h^{2+\varepsilon}, h^{\varepsilon})$ on the entire time interval $[0,T]$.

Notice that $(a + i b) \partial^2_u \varphi = \partial_u( (a + i b) \partial_u \varphi) -  (\partial_u a + i \partial_u b) \partial_u \varphi = (A + i B)\varphi -  (\partial_u a + i \partial_u b) \partial_u \varphi$. As $\partial_u a, \partial_u b \in h^\varepsilon(S^1)$ and the Banach space $h^{1+\varepsilon}$ is an interpolation space between the Banach spaces $h^{\varepsilon}$ and  $h^{2+\varepsilon}$, the perturbation operator $ \mathscr{A}_1 = -  (\partial_u a + i \partial_u b) \partial_u : h^{2+\varepsilon} \to h^\varepsilon$ has the relative zero norm, i.e. for any $\delta>0$ there exists a constant $C_\delta>0$ such that $\Vert \mathscr{A}_1 \varphi \Vert_{h^\varepsilon} \le \delta \Vert \varphi \Vert_{h^{2+\varepsilon}} +C_\delta \Vert \varphi \Vert_{h^{\varepsilon}}$ for each $\varphi\in h^{2+\varepsilon}$. Here we have used the interpolation inequality (\ref{interpolation}). Since the class of linear operators belonging to the maximal regularity class is closed with respect to perturbations with the zero relative norm (c.f. \cite[Lemma 2.5]{Angenent1990}), we conclude that the operator $-(a + i b) \partial^2_u $ belongs to the maximal regularity class ${\mathcal M}(h^{2+\varepsilon}, h^{\varepsilon})$ on the time interval $[0,T]$.

If we denote
\[
{\mathcal Q}=\left(
\begin{array}{ccc}
 T_1 & T_1 T_2+i T_3 & T_1
   T_2-i T_3 \\
 T_2 & -T_1^2-T_3^2 & -T_1^2-T_3^2
   \\
 T_3 & T_2 T_3-i T_1 & T_2 T_3 + i T_1
\end{array} ,
\right)
\]
then ${\mathcal Q}$ is a similarity matrix such that ${\mathcal Q}^{-1} {\mathcal A} {\mathcal Q} = {\mathcal D}$ where ${\mathcal D} = diag(\mu_1, \mu_2, \mu_3)$, $\mu_1=a, \mu_2=a-i b, \mu_3 = a+ i b$. Note that the matrix ${\mathcal Q}={\mathcal Q}(\mb{T})$ analytically depend on the vector $\mb{T}\in\mathbb{R}^3$.

For given $0<\varepsilon <1$ and $k=0, \frac12, 1$ we define the following scale of Banach spaces of H\"older continuous functions defined on the periodic domain $S^1$:
\begin{equation}
E_k = h^{2k +\varepsilon}(S^1)\times h^{2k +\varepsilon}(S^1) \times h^{2k +\varepsilon}(S^1).
\label{Espaces}
\end{equation}

\begin{proposition}\label{maximalregularity}
Assume $a,b\in h^{1+\varepsilon}(S^1)$ and the function $a$ is strictly positive, $a>0$ . Let $T>0$. Then
\begin{enumerate}
    \item the operator $-(a \pm i b) \partial^2_u $ belongs to the maximal regularity class ${\mathcal M}(h^{2+\varepsilon}(S^1), h^{\varepsilon}(S^1))$ on the time interval $[0,T]$,
\item if $\mb{T} \in E_{\frac12}, |\mb{T}|=1$, then the linear operator ${\mathcal A}(a,b,\mb{T}) \partial^2_u  = (a I + b [\mb{T}]_\times) \partial^2_u $  belongs to the maximal regularity class ${\mathcal M}(E_1, E_0)$ on the time interval $[0,T]$.
\end{enumerate}

\end{proposition}

\subsection{Local existence and uniqueness of H\"older smooth solutions}

Let us denote $\mb{X}$ the vector of parametrizations belonging to the Banach space $\mc{E}_{k}$
\[
\mb{X} = (\mb{X}^1, \ldots, \mb{X}^n) \in \mc{E}_{k}, \quad \text{where} \
\mc{E}_{k}= \underbrace{E_{k}\times\ldots\times E_{k}}_{n - times}, \quad k=0,\ 1/2,\ 1.
\]
Clearly, we have the following continuous and compact embedding: $\mc{E}_{1}\hookrightarrow \mc{E}_{1/2}\hookrightarrow \mc{E}_{0}$.

Now, let us define the  mapping $\mathscr{F}_0:\mc{E}_{1} \to \mc{E}_{0}$ as the principal part of the evolution equation (\ref{abstract-form}), i.e.
$\mathscr{F}^i_0(\mb{X}) = {\mathcal A}(a^i,b^i,\partial_{s^i}\mb{X}^i) \partial^2_{s^i} \mb{X}^i$. To prove local existence and uniqueness of solutions we employ the so-called uniform tangential redistribution velocity  defined in Section~\ref{redistribution}. If $\alpha^i$ is such that the total tangential redistribution $v_T^i = \mb{F}^i\cdot  \partial_{s^i}\mb{X}^i + \alpha^i$, then $g^i(u,t)=|\partial_u\mb{X}^i| = L(\Gamma^i_t)$ provided that the initial curve $\Gamma^i_0$ is parametrized uniformly, i.e. $g^i(u,0)= L(\Gamma^i_0)$ for each $u\in [0,1], i=1,\ldots, n$ (see Remark~\ref{uniform}). Hence
\[
d s^i = L(\Gamma^i) du, \ u\in [0,1], \ s^i\in [0, L(\Gamma^i)].
\]
With this parametrization the operator $\mathscr{F}^i_0(\mb{X})$ can be rewritten as follows:
\[
\mathscr{F}^i_0(\mb{X}) = L(\Gamma^i)^{-2} {\mathcal A}(a^i,b^i,\partial_{s^i}\mb{X}^i) \partial^2_{u} \mb{X}^i.
\]
Further, we define the nonlocal  mapping $\mathscr{F}_1:\mc{E}_{1/2} \to \mc{E}_{0}$ as follows:
\[
\mathscr{F}^i_1(\mb{X}) = \mb{F}^i(\mb{X}^i, \partial_{s^i}\mb{X}^i, \gamma^{i1}, \ldots, \gamma^{in}),
\]
where $\mb{X}\in \mc{E}_{1/2}$ and the interaction terms  are defined as in (\ref{eq:F}), i.e.
\[
\gamma^{ij}(\mb{X}^i, \Gamma^j) = \int_{\Gamma^j} f^{ij}(\mb{X}^i, \partial_{s^i}\mb{X}^i, \mb{X}^j, \partial_{s^j}\mb{X}^j) ds^j.
\]
Finally, we define the tangential part $\mathscr{F}_2$ of equation (\ref{abstract-form}), i.e. $\mathscr{F}^i_2(\mb{X}^i) = \alpha^i \partial_{s^i}\mb{X}^i$. Concerning qualitative properties of the functions
$a^i=a^i(\mb{X}^i, \mb{T}^i),  b^i=b^i(\mb{X}^i, \mb{T}^i), \mb{F}^i = \mb{F}^i(\mb{X}^i, \mb{T}^i, \gamma^{i1}, \ldots, \gamma^{in})$ where
$\gamma^{ij}(\mb{X}^i, \Gamma^j) = \int_{\Gamma^j} f^{ij}(\mb{X}^i, \mb{T}^i, \mb{X}^j, \mb{T}^j) ds^j$
we will assume the following structural hypothesis:
\begin{equation}
\left\{
\begin{aligned}
 & a^i, b^i: \mathbb{R}^3\times \mathbb{R}^3 \to \mathbb{R}, \quad a^i\ge \underline{a} >0,
 \\
 & \mb{F}^i:\mathbb{R}^3\times \mathbb{R}^3 \times \mathbb{R}^n\to \mathbb{R}^3, \quad
 f^{ij}:\mathbb{R}^3\times \mathbb{R}^3 \times \mathbb{R}^3\times \mathbb{R}^3\to \mathbb{R}^3,
 \quad \text{for} \ i,j=1,\ldots, n,
\\
& \text{are $C^2$ smooth and globally Lipschitz continuous functions, $\underline{a}>0$ is a constant.}
\end{aligned}
\right.
\label{smootheness}\tag{H}
\end{equation}

\begin{proposition}\label{Fder}
Assume the hypothesis (\ref{smootheness}) and $\alpha^i, i=1, \ldots, n$, is the tangential velocity preserving the  relative local length. Let $\tilde{\mb{X}}\in \mc{E}_{1}$ be such that $\tilde{g}^i>0$ for each $i=1,\ldots, n$. Then,
\begin{enumerate}
\item The principal part mapping $\mathscr{F}_0:\mc{E}_{1} \to \mc{E}_{0}$ is  $C^1$ differentiable. Its Fr\'echet derivative $\mathscr{F}^\prime_0(\tilde{\mb{X}})$ belongs to the maximal regularity class $\mathcal{M}(\mathcal{E}_1, \mathcal{E}_0)$.

\item The nonlocal mappings $\mathscr{F}_1$ and $\mathscr{F}_2$ are $C^1$ differentiable as mappings from $\mc{E}_{1/2}$ into $\mc{E}_{0}$. The Fr\'echet derivative $\mathscr{F}^\prime_k(\tilde{\mb{X}}), k=1,2$, considered now as a mapping from $\mc{E}_{1}$ into $\mc{E}_{0}$ has the relative zero norm.

\item The total mapping $\mathscr{F}:\mc{E}_{1} \to \mc{E}_{0}$ where $\mathscr{F}= \mathscr{F}_0 + \mathscr{F}_1 + \mathscr{F}_2$ is  $C^1$ differentiable, and $\mathscr{F}^\prime(\tilde{\mb{X}})$ belongs to the maximal regularity class $\mathcal{M}(\mathcal{E}_1, \mathcal{E}_0)$.

\end{enumerate}

\end{proposition}

\begin{proof}

Let $\tilde{\mb{X}}\in E_{1/2}$. Denote $\tilde{s}$ the unit arc-length parametrization of the curve $\tilde{\Gamma}=\{\tilde{\mb{X}}(u), u\in [0,1]\}$. Then $d\tilde{s} = \tilde{g}(u) du$ where $\tilde{g}(u)=|\partial_u\tilde{\mb{X}}|$. The derivative of the local length $g=|\partial_u\mb{X}|$ at the point $\tilde{\mb{X}}\in \mc{E}_{1/2}$ in a direction $\mb{X}\in \mc{E}_{1/2}$ is given by $g'(\tilde{\mb{X}}) \mb{X} = \partial_{\tilde{s}}\tilde{\mb{X}}\cdot \partial_u \mb{X} $. As a consequence, the derivative of the total length functional $L(\Gamma) = \int_\Gamma ds = \int_0^1 |\partial_u\mb{X}| du$ in the direction $\mb{X}\in E_{1/2}$ is given by $L'(\tilde{\Gamma}) \mb{X} = \int_0^1 \partial_{\tilde{s}}\tilde{\mb{X}}\cdot \partial_u \mb{X} du$.

To prove statement 1), we note that the linearization $\mathscr{F}^\prime_0(\tilde{\mb{X}}) \mb{X}$ at the point $\tilde{\mb{X}}$ in the direction $\mb{X}$ has the form:
\[
\mathscr{F}^{i\, \prime}_0(\tilde{\mb{X}}) \mb{X}
=  L(\tilde{\Gamma}^i)^{-2} {\mathcal A}(\tilde{a}^i,\tilde{b}^i,\partial_{\tilde{s}^i}\tilde{\mb{X}}^i) \partial^2_{u} \mb{X}^i + \tilde{\mathcal B} [\mb{X}^i],\ i=1,\ldots, n,
\]
where the linear operator $\tilde{\mathcal B}$ represents lower order terms with respect to differentiation. Namely,
\begin{eqnarray*}
\tilde{\mathcal B} [\mb{X}^i]
&=& L(\tilde{\Gamma}^i)^{-2}
\biggl(
\nabla_{a^i} \tilde{\mathcal A}
\left[
\nabla_{\mb{X}^i}{\tilde{a}^i}\, \mb{X}^i
+
\nabla_{\mb{T}^i}{\tilde{a}^i} \, \partial_{s^i} \mb{X}^i
\right]
+ \nabla_{b^i} \tilde{\mathcal A}
\left[
\nabla_{\mb{X}^i}{\tilde{b}^i} \, \mb{X}^i
+
\nabla_{\mb{T}^i}{\tilde{b}^i} \, \partial_{s^i} \mb{X}^i
\right]
\\
&& + \nabla_{\mb{T}^i} \tilde{\mathcal A} \, \partial_{s^i} \mb{X}^i
- 2  L(\tilde{\Gamma}^i)^{-1} L^\prime(\tilde{\Gamma}^i) \mb{X}^i {\mathcal A}
\biggr) \partial^2_{u} \tilde{\mb{X}}^i,
\end{eqnarray*}
where the coefficients $a^i, b^i$, the mapping $\tilde{\mathcal{A}}$, and their first derivatives are evaluated at $\tilde{\mb{X}}^i$.
With regard to assumption made on coefficients $\tilde{a}^i=a^i(\tilde{\mb{X}}^i, \tilde{\mb{T}}^i)$ and $\tilde{b}^i=b^i(\tilde{\mb{X}}^i, \tilde{\mb{T}}^i)$ we conclude that the lower order linear operator $\tilde{\mathcal B}$ is a bounded linear operator from the Banach space $E_{1/2}$ into $E_0$. As a consequence, it has the zero relative norm if considered as a mapping from $E_1$ into $E_0$.

Since $L(\tilde{\Gamma}^i)>0$ is a positive constant then, according to Proposition~\ref{maximalregularity}, part 2), the linear operator $L(\tilde{\Gamma}^i)^{-2} {\mathcal A}(\tilde{a}^i,\tilde{b}^i,\tilde{\mb{T}}^i) \partial^2_u$  belongs to the maximal regularity class ${\mathcal M}(E_1, E_0)$ on the time interval $[0,T]$. Therefore, the linearization $\mathscr{F}^{i\, \prime}_0(\tilde{\mb{X}})$ belongs to the maximal regularity class ${\mathcal M}(E_1, E_0)$ because the class ${\mathcal M}(E_1, E_0)$ is closed with respect to perturbation with relative zero norm (c.f.  \cite[Lemma 2.5]{Angenent1990}, DaPrato and Grisvard \cite{daprato}, Lunardi \cite{Lunardi1984}). Hence, $\mathscr{F}^\prime_0(\tilde{\mb{X}})$ belongs to the maximal regularity pair $\mathcal{M}(\mathcal{E}_1, \mathcal{E}_0)$, as claimed.

In order to prove part 2) of the proposition, we first evaluate the derivative of the nonlocal function $\gamma^{ij}$ at the point $\tilde{\mb{X}}^i$ in the direction $\mb{X}^i$. We have
\[
\gamma^{ij\, \prime}_{\mb{X}^i}(\tilde{\mb{X}}^i, \tilde{\Gamma}^j) \mb{X}^i
= \int_{\tilde{\Gamma}^j} \left(
\tilde{f}^{ij\, \prime}_{\mb{X}^i} \mb{X}^i
+
\tilde{f}^{ij\, \prime}_{\mb{T}^i}
\left[
L(\tilde{\Gamma}^i)^{-1} \partial_u \mb{X}^i -  ( L(\tilde{\Gamma}^i)^{-2} L^\prime(\tilde{\Gamma}^i) \mb{X}^i ) \partial_u \tilde{\mb{X}}^i
\right]
\right) d\tilde{s}^j,
\]
\begin{eqnarray*}
\gamma^{ij\, \prime}_{\mb{X}^j}(\tilde{\mb{X}}^i, \tilde{\Gamma}^j) \mb{X}^j
&=& \int_{\tilde{\Gamma}^j} \left(
\tilde{f}^{ij\, \prime}_{\mb{X}^j} \mb{X}^j
+
\tilde{f}^{ij\, \prime}_{\mb{T}^j}
\left[
L(\tilde{\Gamma}^j)^{-1} \partial_u \mb{X}^j -  ( L(\tilde{\Gamma}^j)^{-2} L^\prime(\tilde{\Gamma}^j) \mb{X}^j ) \partial_u \tilde{\mb{X}}^j
\right]
\right) d\tilde{s}^j
\\
&& +
\int_{\tilde{\Gamma}^j} \tilde{f}^{ij} L(\tilde{\Gamma}^j)^{-1} L^\prime(\tilde{\Gamma}^j) \mb{X}^j d\tilde{s}^j.
\end{eqnarray*}
Here we have used the fact that the directional derivative of the tangent vector $\mb{T}^i = \partial_{s^i} \mb{X}^i = L(\Gamma^i)^{-1} \mb{X}^i$ in the direction $ \mb{X}^i$ is given by $L(\tilde{\Gamma}^i)^{-1} \partial_u \mb{X}^i -  ( L(\tilde{\Gamma}^i)^{-2} L^\prime(\tilde{\Gamma}^i) \mb{X}^i ) \partial_u \tilde{\mb{X}}^i$. It means that that the mapping $\gamma^{ij}$ is $C^1$ differentiable as a mapping from $E_{1/2}\times E_{1/2} \to \mathbb{R}$ and its derivative is a bounded linear operator from $E_{1/2}\times E_{1/2}$ into $\mathbb{R}$. Hence the linearization  $\mathscr{F}^\prime_1(\tilde{\mb{X}})$ is a bounded linear operator from the Banach space $\mathcal{E}_{1/2}$ into $\mathcal{E}_0$.

Finally, let us consider the tangential part $\mathscr{F}_2(\mb{X})$  where $\mathscr{F}^i_2(\mb{X}^i) = \alpha^i \partial_{s^i}\mb{X}^i, i=1,\ldots, n$. Recall that the uniform tangential redistribution $\alpha^i = v^i_T-\mb{F}^i \cdot \mb{T}^i$ is computed from the equation $\partial_{s^i} v^i_T= \kappa^i v^i_N  -\frac{1}{L(\Gamma^i)} \int_{\Gamma^i} \kappa v^i_N ds^i$, see (\ref{unifalpha}). Let us denote the auxiliary function $\psi(\mb{X}^i) = \kappa^i v^i_N$. Since $v^i_N = a^i \kappa^i + \mb{F}^i \cdot \mb{N}^i$ then using the Frenet-Serret formula
$\partial^2_{s^i} \mb{X}^i = \partial_{s^i} \mb{T}^i =  \kappa^i \mb{N}^i$ and the fact that $d s^i = L(\Gamma^i) du$ we obtain
\[
\psi(\mb{X}^i) = a^i (\kappa^i)^2 + \mb{F}^i \cdot  \kappa^i \mb{N}^i =  a^i |\partial^2_{s^i} \mb{X}^i|^2 + \mb{F}^i \cdot \partial^2_{s^i} \mb{X}^i
=
L(\Gamma^i)^{-2}\left(
a^i |\partial^2_{u} \mb{X}^i|^2 + \mb{F}^i \cdot \partial^2_{u} \mb{X}^i
\right).
\]
Let $0<\varepsilon^\prime < \varepsilon$ and $E^\prime_k = h^{2k +\varepsilon^\prime}(S^1)\times h^{2k +\varepsilon^\prime}(S^1) \times h^{2k +\varepsilon^\prime}(S^1)$. Clearly, $E^\prime_k \hookrightarrow E_k$ and $E^\prime_1$ is an interpolation space between $E_0$ and $E_1$. The mapping $\psi: E^\prime_1 \to E^\prime_0$ is $C^1$ differentiable and its derivative $\psi^\prime(\tilde{\mb{X}^i})$ is a bounded linear operator from  $E^\prime_1$ to $h^{\varepsilon^\prime}(S^1)$. As a consequence, the mapping $\mb{X}^i \mapsto  \kappa^i v^i_N  -\frac{1}{L(\Gamma^i)} \int_{\Gamma^i} \kappa^i v^i_N ds^i $ is $C^1$  differentiable as a mapping from the Banach space $E^\prime_1$ into $h^{\varepsilon^\prime}(S^1)$. Since the total velocity $v^i_T$ is an integral of this mapping we obtain $\mb{X}^i \mapsto v^i_T$ as well as  $\mb{X}^i \mapsto \alpha^i =  v^i_T-\mb{F}^i \cdot \partial_{s^i} \mb{X}^i$ is $C^1$ differentiable as a mapping from the space $E^\prime_1$ into $h^{1+\varepsilon^\prime}(S^1) \hookrightarrow h^{\varepsilon}(S^1)$. Hence the mapping $\mathscr{F}^i_2$ (now considered as a mapping from $E_1$ into $E_0$) is $C^1$ differentiable and its linearization $\mathscr{F}^{i\prime}_2(\tilde{\mb{X}})$ has zero relative norm.

The last statement 3) of the proposition now follows as the class ${\mathcal M}(E_1, E_0)$ is closed with respect to perturbations with relative zero norm (c.f. Angenent \cite{Angenent1990, Angenent1990b}, DaPrato and Grisvard \cite{daprato}, Lunardi \cite{Lunardi1984}).
\end{proof}

Now we can state the following result on local existence, uniqueness and continuation of solutions.

\begin{theorem}\label{theo-main}
Assume the hypothesis (\ref{smootheness}) and $\alpha^i, i=1, \ldots, n$, is the tangential velocity preserving the relative local length. Assume the parametrization $\mb{X}_0\equiv (\mb{X}^i_0)_{i=1}^n,$ of initial curves $\Gamma^i_0$ belongs to the H\"older space $\mc{E}_1$,  and it is uniform parametrization, i.e. $|\partial_u\mb{X}^i_0(u)| = L(\Gamma^i_0) >0$ for all $u\in I$ and $i=1,\ldots, n$. Assume the functions $a^i, b^i, \mb{F}^i, f^{ij}$ satisfy the assumptions (\ref{smootheness}).

Then there exists $T>0$ and the unique family of curves  $\{\Gamma^i_t, t\in[0,T]\}, i=1,\ldots, n$, evolving in 3D according to the system of nonlinear nonlocal geometric equations:
\begin{equation}
\partial_t\mb{X}^i = a^i \partial^2_{s^i} \mb{X}^i + b^i (\partial_{s^i}\mb{X}^i\times\partial^2_{s^i} \mb{X}^i)  + \mb{F}^i +\alpha^i \mb{T}^i, \quad i=1,\ldots, n,
\label{govequation}
\end{equation}
such that their parametrization satisfies $\mb{X}=(\mb{X}^i)_{i=1}^n \in  C([0,T], \mc{E}_1) \cap C^1([0,T], \mc{E}_0)$, and $\mb{X}(\cdot, 0)= \mb{X}_0$. Furthermore, if the maximal time of existence $T_{max}<\infty$ is finite then
\[
\uplim_{t\to T_{max}} \max_{i, \Gamma^i_t} |\kappa^i(\cdot, t)| = \infty.
\]

\end{theorem}

\begin{proof}
The proof follows from the abstract result on existence and uniqueness of solutions to (\ref{govequation}) due to Angenent \cite{Angenent1990}. It is based on the linearization of the abstract evolution equation (\ref{abstratF})
\[
\partial_t \mb{X} + \mathscr{F}(\mb{X}) = 0, \quad \mb{X}(0) = \mb{X}_0,
\]
in the Banach space $\mc{E}_1$. With regard to Proposition~\ref{Fder}, for any $\tilde{\mb{X}}$, the linearization $\mathscr{F}'(\tilde{\mb{X}} )$ generates an analytic semigroup and it belongs to the maximal regularity class $\mc{M}(\mc{E}_1, \mc{E}_0)$ of linear operators from  the Banach space $\mathcal{E}_1$ into Banach space $\mathcal{E}_0$. The local existence and uniqueness of a solution $\mb{X}=(\mb{X}^i)_{i=1}^n \in  C([0,T], \mc{E}_1) \cap C^1([0,T], \mc{E}_0)$, and $\mb{X}(\cdot, 0)= \mb{X}_0$ now follows from the abstract result \cite[Theorem 2.7]{Angenent1990} due to  Angenent.

In order to prove the last statement we use a simple bootstrap argument. Suppose that the maximal time of existence is finite and $\max_{i, \Gamma^i_t} |\kappa^i(\cdot, t)|<\infty$. Then the solution $\mb{X}$ belongs to the space $C([0,T], \mc{E}_1) \cap C^1([0,T], \mc{E}_0)$ for any compact subinterval $[0,T]\subset [0, T_{max})$. Since $\kappa^i$ is bounded so does the second derivative $\partial^2_{s^i} \mb{X}^i = \partial_{s^i} \mb{T}^i = \kappa^i \mb{N}^i$. It means that the lower order terms in the governing equation are continuous and uniformly bounded. That is the function $\tilde{g}^i = \mb{F}^i +\alpha^i \mb{T}^i$ belongs to the space $C([0,T_{max}], E_0)$, and the solution $\mb{X}^i$ satisfies the linear evolution equation
\begin{equation}
\partial_t\mb{X}^i = \tilde{\mc{A}}^i \partial^2_{s^i} \mb{X}^i + \tilde{g}^i, \qquad \mb{X}^i(\cdot, 0) = \mb{X}^i_0 \in \mc{E}_1,
\label{linear}
\end{equation}
where $\tilde{\mc{A}}^i(\cdot, t) = \mc{A}(a^i(\cdot, t), b^i(\cdot, t),  \partial_{s^i} \mb{X}^i(\cdot, t))$ with $a^i(\cdot, t) =a^i(\mb{X}^i(\cdot, t), \mb{T}^i(\cdot, t))$ and $b^i(\cdot, t) =b^i(\mb{X}^i(\cdot, t), \mb{T}^i(\cdot, t))$ is a time dependent matrix belonging to the space $C([0,T_{max}], \mc{E}_{1/2})$. Applying the maximal regularity for the linear equation (\ref{linear})  we conclude that the solution $\mb{X}^i \in C([0,T_{max}], E_1) \cap C^1([0,T_{max}], E_0)$. It means that $\mb{X} \in C([0,T_{max}], \mc{E}_1) \cap C^1([0,T_{max}], \mc{E}_0)$, and so we can continue a solution beyond the maximal time of existence $[0, T_{max})$ starting from the initial condition $\mb{X}(\cdot, 0) = \mb{X}(\cdot, T_{max})\in \mc{E}_1$, a contradiction. Therefore $T_{max}=\infty$, as claimed, provided that the curvatures $\kappa^i, i=1,\ldots, n$, remain bounded on the maximal time of existence $[0, T_{max})$.
\end{proof}

\begin{remark}\label{gen-hypothesis}
The structural hypothesis (\ref{smootheness}) can be slightly relaxed in the case when the initial curves do not intersect each other. 

We assume there exist open non-intersecting neighborhoods ${\mathcal O}^i\in \mathbb{R}^3$ of  initial curves $\Gamma^i_0\subset {\mathcal O}^i, i=1,  \ldots, n$ such that  ${\mathcal O}^i\cap {\mathcal O}^j=\emptyset$ for $i\not= j$, and the following structural assumptions hold: 
\begin{equation}
\left\{
\begin{aligned}
 & a^i, b^i: {\mathcal O}^i\times \mathbb{R}^3 \to \mathbb{R}, \quad a^i\ge \underline{a} >0,
 \\
 & \mb{F}^i:{\mathcal O}^i\times \mathbb{R}^3 \times \mathbb{R}^n\to \mathbb{R}^3, \quad
 f^{ij}:{\mathcal O}^i\times \mathbb{R}^3 \times {\mathcal O}^j\times \mathbb{R}^3\to \mathbb{R}^3,
 \quad \text{for} \ i,j=1,\ldots, n,
\\
& \text{are $C^2$ smooth and globally Lipschitz continuous functions, $\underline{a}>0$ is a constant.}
\end{aligned}
\right.
\label{smootheness2}\tag{H'}
\end{equation}
If we replace the hypothesis (\ref{smootheness}) by its generalization (\ref{smootheness2}) then the local existence result stated in the main Theorem~\ref{theo-main} remains true except of the limiting behavior as $t\to T_{max}$ where the last statement in Theorem~\ref{theo-main} should be replaced as follows: either $\uplim_{t\to T_{max}} \max_{i, \Gamma^i_t} |\kappa^i(\cdot, t)| = \infty$, or $\lim_{t\to T_{max}} \min_{i} \hbox{dist} (\Gamma^i_t, \partial {\mathcal O}^i) = 0$. Here $\hbox{dist} (\Gamma^i_t, \partial {\mathcal O}^i)$ is the distance between $\Gamma^i_t$ and the boundary $\partial {\mathcal O}^i$ of the neighborhood ${\mathcal O}^i$. 

The hypothesis (\ref{smootheness2}) can be employed in examples involving flows of non-intersecting curves driven by normal and binormal velocity under the Biot-Savart law (\ref{biot-savart-force}).

\end{remark}

\section{Numerical discretization scheme based on the method of lines}\label{Sec:Num}

In this section we present a numerical discretization scheme for solving the system of equations (\ref{abstract-form-general}) enhanced by the tangential velocity $\alpha^i$. Our discretization scheme is based on the method of lines with the spatial discretization obtained by means of the finite volume method. For simplicity, we consider one evolving curve $\Gamma$ (omitting the curve index $i$) and  rewrite the abstract form of (\ref{abstract-form-general}) in terms of the principal parts of its velocity
\begin{equation}
    \label{eq:num_start}
    \partial_t \mb{X} = a \partial_s^2 \mb{X}
    + b(\partial_s \mb{X} \times \partial_s^2 \mb{X})
    + \mb{F}
    + \alpha \mb{T}.
\end{equation}
We place $M$ discrete nodes $\mb{x}_k = \mb{X}(u_k)$,  $k = 0,1,2, \ldots, M$ along the curve $\Gamma$. Corresponding dual nodes are defined as $\mb{x}_{k \pm \frac12} = \mb{X}(u_{k \pm \frac12})$  (see Figure \ref{FVMfig}). Here $ u_{k \pm \frac12} = u_k \pm \frac{h}2$, where $h=1/M$, and $(\mb{x}_k + \mb{x}_{k+1}) / 2$ denote averages on segments connecting nearby discrete nodes and differs from $\mb{x}_{k \pm \frac12} \in \Gamma$. The $k$-th segment $\mathcal{S}_k$ of $\Gamma$ between the nodes $\mb{x}_{k-1}$ and $\mb{x}_{k}$ represents the finite volume.
Integration of equation (\ref{eq:num_start}) over the segment of $\Gamma$ between the nodes $\mb{x}_{k + \frac12}$ and $\mb{x}_{k - \frac12}$ yields

\begin{equation}
\label{eq:num_int}
\begin{split}
\int_{u_{k-\frac12}}^{u_{k+\frac12}} \partial_t \mb{X} |\partial_u \mb{X}| d u = &
\int_{u_{k-\frac12}}^{u_{k+\frac12}} a\frac{\partial}{\partial_u} \left( \frac{\partial_u \mb{X}}{|\partial_u \mb{X}|} \right) d u
+
\int_{u_{k-\frac12}}^{u_{k+\frac12}}
b (\partial_s \mb{X} \times \partial_s^2 \mb{X}) |\partial_u \mb{X}| d u
 \\
+ & \int_{u_{k-\frac12}}^{u_{k+\frac12}} \mb{F} |\partial_u \mb{X}| d u
+ \int_{u_{k-\frac12}}^{u_{k+\frac12}} \alpha \partial_u \mb{X} d u.
\end{split}
\end{equation}
Let us denote $d_k = |\mb{x}_k - \mb{x}_{k-1}|$ for $k=1,2,\ldots,M,M+1$, where $\mb{x_0} = \mb{x}_M$ and $\mb{x}_1 = \mb{x}_{M+1}$ for closed curve $\Gamma$ and we approximate the integral expressions in (\ref{eq:num_int}) by means of the finite volume method along $\Gamma$ as follows:
\begin{equation}
\begin{split}
 \int_{u_{k-\frac12}}^{u_{k+\frac12}} \partial_t \mb{X} |\partial_u \mb{X}| d  u
& \approx \frac{d  \mb{x}_k}{d t} \frac{d_{k+1} + d_k}{2},
\\
 \int_{u_{k-\frac12}}^{u_{k+\frac12}}
a \partial_u \left( \frac{\partial_u \mb{X}}{|\partial_u \mb{X}|} \right) d u
&\approx
a_k \left(\frac{\mb{x}_{k+1} - \mb{x}_k}{d_{k+1}} - \frac{\mb{x}_k - \mb{x}_{k-1}}{d_k}\right),
\\
 \int_{u_{k-\frac12}}^{u_{k+\frac12}}
b (\partial_s \mb{X} \times \partial_s^2 \mb{X}) |\partial_u \mb{X}| d u
&\approx
b_k \frac{d_{k+1} + d_k}{2} \kappa_k (\mb{T}_k \times \mb{N}_k),
\\
\int_{u_{k-\frac12}}^{u_{k+\frac12}} \mb{F} |\partial_u \mb{X}| d u
&\approx  \mb{F}_k \frac{d_{k+1} + d_{k}}{2},
\\
\int_{u_{k-\frac12}}^{u_{k+\frac12}} \alpha \partial_u \mb{X} d u
&\approx  \alpha_k \frac{\mb{x}_{k+1} - \mb{x}_{k-1}}{2}.
\end{split}
\label{discretization}
\end{equation}
The approximation of the nonnegative curvature $\kappa$, tangent vector $\mb{T}$ and normal vector $\mb{N}$, $\kappa \mb{N} = \partial_s \mb{T}$ read as follows:

\begin{equation}
\begin{split}
\label{krivost}
\kappa_k & \approx  \left| \mb{T}_k \times \frac{2}{d_k + d_{k+1}}
\left(\frac{\mb{x}_{k+1} - \mb{x}_k}{d_{k+1}} - \frac{\mb{x}_k - \mb{x}_{k-1}}{d_k}\right)
\right| \\
& =
\left|\frac{2}{d_k + d_{k+1}}
\left(\frac{\mb{x}_{k+1} - \mb{x}_k}{d_{k+1}} - \frac{\mb{x}_k - \mb{x}_{k-1}}{d_k}\right)
\right|,
\\
\mb{T}_k & \approx \frac{\mb{x}_{k+1} - \mb{x}_{k-1}}{d_{k+1}+d_k}, \quad
\mb{N}_k \approx \kappa_k^{-1} \frac{2}{d_k + d_{k+1}}
\left(\frac{\mb{x}_{k+1} - \mb{x}_k}{d_{k+1}} - \frac{\mb{x}_k - \mb{x}_{k-1}}{d_k}\right).
\end{split}
\end{equation}
Here and hereafter, we assume $\partial_t \mb{X}, \partial_u \mb{X}, \mb{F}$ and $\alpha$ are constant over the finite volume between the nodes $\mb{x}_{k + \frac12}$ and $\mb{x}_{k - \frac12}$, taking values $\partial_t \mb{X}_k, \partial_u \mb{X}_k, \mb{F}_k$ and $\alpha_k$, respectively.
In approximation $\mb{F}_k$ of the nonlocal vector valued function $\mb{F}$, we assume the curve $\Gamma$ entering the definition of $\mb{F}$ is approximated by the polygonal curve with vertices $(\mb{x}_0, \mb{x}_1, \ldots, \mb{x}_M)$.
In order to find the approximation $\alpha_k$ of the tangential velocity given by equations (\ref{eq:alpha}) and (\ref{alpha-asymptotic}) we apply a simple integration rule and obtain the following formula
\begin{equation}
    \label{eq:alpha_num}
    \alpha_k \approx -\mb{F}_k \cdot \mb{T}_k + \mb{F}_0  \cdot \mb{T}_0 + \alpha_0
    +
    \sum_{j=1}^k \kappa_j v_{N,j} d_j -  \frac{\sum_{j=1}^k d_j}{L}
    \sum_{j=1}^M \kappa_j v_{N,j} d_j  + \omega  \sum_{j=1}^k \left( \frac{L}{M} - d_j \right),
\end{equation}
for $k=1,2, \ldots, M$, where $L=\sum_{j=1}^M d_j$ is the total length of the curve and $\omega\ge 0$ is a redistribution parameter. Here the discrete normal velocity $v_{N,j}$ is given by
\[
v_{N,j} = a\, \kappa_j + \mb{F}_j \cdot \mb{N}_j.
\]
The values  $\alpha_0=\alpha_M$ are chosen in such a way that $\sum_{j=1}^M  \alpha_j d_j =0$. If $\omega=0$ we obtain the uniform redistribution. If $\omega>0$ we obtain asymptotically uniform redistribution (see (\ref{alpha-asymptotic})). In summary, the semi-discrete scheme for solving (\ref{eq:num_start}) can be written as follows
\begin{eqnarray}
\frac{d  \mb{x}_k}{d t} \frac{d_{k+1}+d_k}{2}
&=&
a_k \left(\frac{\mb{x}_{k+1} - \mb{x}_k}{d_{k+1}} - \frac{\mb{x}_k - \mb{x}_{k-1}}{d_k}\right)
 + b_k \frac{d_{k+1}+d_k}{2} \kappa_k (\mb{T}_k \times \mb{N}_k)
\label{eq:DCScheme}
\\
&& + \mb{F}_k \frac{d_{k+1} + d_{k}}{2}
 + \alpha_k \frac{\mb{x}_{k+1} - \mb{x}_{k-1}}{2},
 \nonumber
\\
\mb{x}_k(0) &=& \mb{X}_{ini}(u_k), \quad \text{for} \ k = 1, \ldots, M.
\label{eq:DCScheme+1}
\end{eqnarray}
Resulting system (\ref{eq:DCScheme}--\ref{eq:DCScheme+1}) of ODEs  is solved numerically by means of the 4th-order explicit Runge-Kutta-Merson scheme with automatic time stepping control and the tolerance parameter $10^{-3}$ (see \cite{0965-0393-24-3-035003}). We chose the initial time-step as $4h^2$, where $h=1/M$ is the spatial mesh size.

\begin{figure}
\begin{center}
\includegraphics[width=0.4\textwidth]{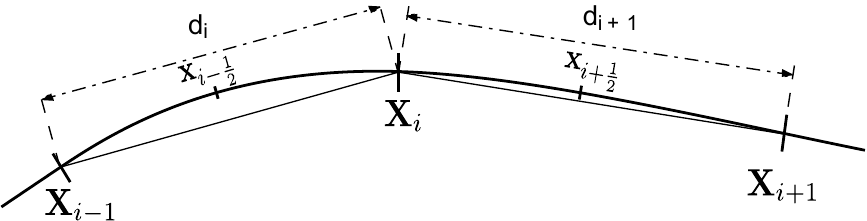}
\end{center}
\caption{Discretization of a segment of a curve by means of the flowing finite volumes.
}
\label{FVMfig}
\end{figure}

\begin{figure}
\begin{center}
\includegraphics[width=0.35\textwidth]{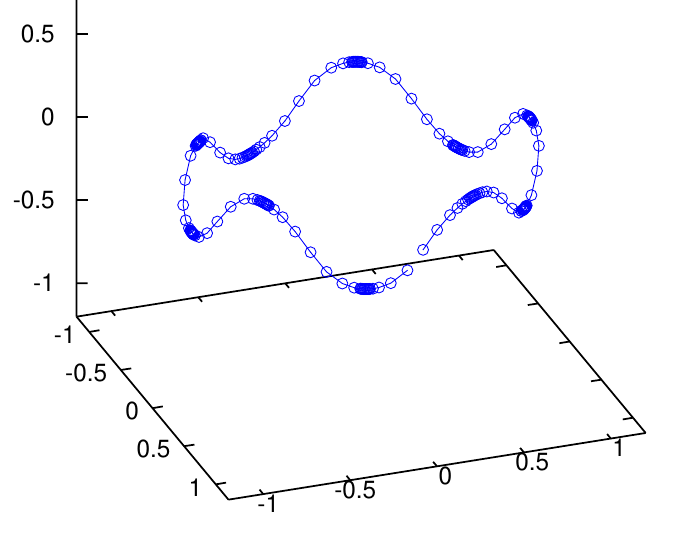}
\includegraphics[width=0.35\textwidth]{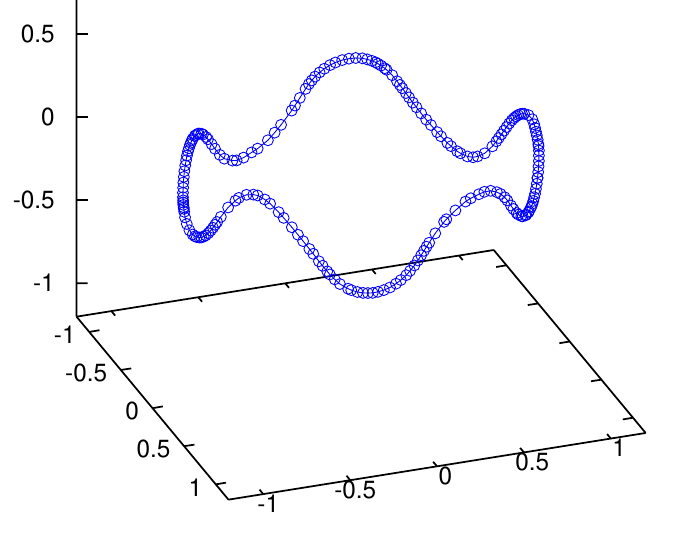}
\end{center}
\caption{Illustration of importance of a suitable choice of the tangential redistribution. Left: no tangential redistribution. Right: tangential redistribution preserving the relative local length (i.e. $\omega = 0$). }
\label{fig:redis}
\end{figure}

\section{Numerical results}

In this section we present several examples of evolution of interacting curves in 3D. Nonlocal interactions between curves are modeled by means of the Biot-Savart law. Subsection~\ref{normalbinormal} is devoted to the motion of interacting curves with a nontrivial normal velocity component $a^i>0$. We apply numerical scheme based on the finite volume approximation of spatial derivatives in combination with the method of lines. In subsection~\ref{binormalonly} we present examples of evolving interacting curves with the binormal velocity with $b^i=1, a^i=0$, and nonlocal interactions. The problem can be reduced to a solution of the system of ODEs and the solution can be represented in terms of evolving concentric circles.

\subsection{Computational examples of 3D curve dynamics under normal and binormal velocity}\label{normalbinormal}

Below, we describe computational examples performed by scheme (\ref{eq:DCScheme}-\ref{eq:DCScheme+1}) designed in Section \ref{Sec:Num}. The examples demonstrate mutual interaction of a pair of closed spatial curves moving according to the motion law (\ref{biot-savart}--\ref{biot-savart-force}) where the interaction force of the Biot-Savart type is used. The semi-discrete scheme (\ref{eq:DCScheme}-\ref{eq:DCScheme+1}) is solved by the fourth-order Runge-Kutta-Merson method with automatic time step control (see, e.g. as in \cite{0965-0393-24-3-035003}), with the tolerance $10^{-3}$.

\medskip

\noindent {\bf Example 1} shows the evolution of two mutually interacting curves - see Figure~\ref{fig:theory1}. Their initial shape is circular with a vertical sinusoidal perturbation, their barycenters are vertically in different planes and horizontally shifted. As it can be seen from the time evolution, the curves exhibit the "frog leap" dynamics (see \cite{Meleshko:12}) - the smaller curves moves vertically through the interior of the larger curve, becomes larger and the process repeats several times until one of them shrinks to a point as a consequence of the normal component of the flow. This example is set as follows - the flow parameters combining the normal and binormal directions are $a^{1,2} = 0.05$ and $b^{1,2} = 0.1$. The initial curves are parametrized as:
%
\begin{equation*}
\mb{X}^1(u,0) =
\left(
\begin{array}{c}
\cos(2 \pi u)+ 0.1  \\
\sin(2 \pi u)       \\
0.2 + 0.2 \sin(6 \pi u)
\end{array}
\right),
\quad
\mb{X}^2(u,0) =
\left(
\begin{array}{c}
       3 \cos(2 \pi u)  \\
 0.1 + 3 \sin(2 \pi u)  \\
-0.2 +0.2 \sin(12 \pi u)
\end{array}
\right),
\quad u \in (0,1).
\end{equation*}
The initial curves do not intersect each other. As an external forcing term we considere the Biot-Savart law (\ref{biot-savart-force}), i.e. we choose $\delta=0$ in (\ref{biot-savart-force-regularized}).
The spatial parametrization is discretized by $M = 100$ segments. The output time step was $\Delta t = 0.2$.

\medskip
\noindent {\bf Example 2} shows the evolution of two mutually interacting curves - see Figure~\ref{fig:theory2}. Their initial configuration consists of two circles in mutually perpendicular planes. In the time evolution, the curves become distorted by the mutual forces and move away each from other.
This example is set as follows - the flow parameters combining the normal and binormal directions are $a^{1,2} = 0.05$ and $b^{1,2} = 0.1$. The initial curves are parameterized as
%
\begin{equation*}
\mb{X}^1(u,0) =
\left(
\begin{array}{c}
2 \cos(2 \pi u)     \\
2 \sin(2 \pi u)     \\
0.0
\end{array}
\right),
\quad
\mb{X}^2(u,0) =
\left(
\begin{array}{c}
2 \sin(2 \pi u)  \\
3.0 \\
2 \cos(2 \pi u)
\end{array}
\right),
\quad u \in (0,1).
\end{equation*}
Again we considered the Biot-Savart law (\ref{biot-savart-force}) as an external forcing term.
The spatial parametrization is discretized by $M = 100$ segments. The output time step was $\Delta t = 0.2$.

\medskip

\noindent {\bf Example 3} shows the evolution of two mutually interacting curves - see Figure~\ref{fig:theory3}. Their initial shape is circular, their barycenters are vertically in different planes and horizontally shifted. In the time evolution, the curves exhibit acrobatic motion when the smaller curve squeezes into the interior of the larger one and loops over it repeatedly.
This example is set as follows - the flow parameters combining the normal and binormal directions are $a^{1,2} = 0.05$ and $b^{1,2} = 0.1$. The initial curves are parametrized as
%
\begin{equation*}
\mb{X}^1(u,0) =
\left(
\begin{array}{c}
\cos(2 \pi u)     \\
\sin(2 \pi u)     \\
0.0
\end{array}
\right),
\quad
\mb{X}^2(u,0) =
\left(
\begin{array}{c}
2 \cos(2 \pi u)  \\
0.5 + 2 \sin(2 \pi u)  \\
1.5
\end{array}
\right),
\quad u \in (0,1).
\end{equation*}
The parametric space is discretized by $M = 150$ segments. The output time step was $\Delta t = 0.2$. The numerical algorithm is stabilized by tangential redistribution.

\begin{figure}
\begin{center}
\includegraphics[width=0.49\textwidth, trim= 1cm 0 1cm  0]{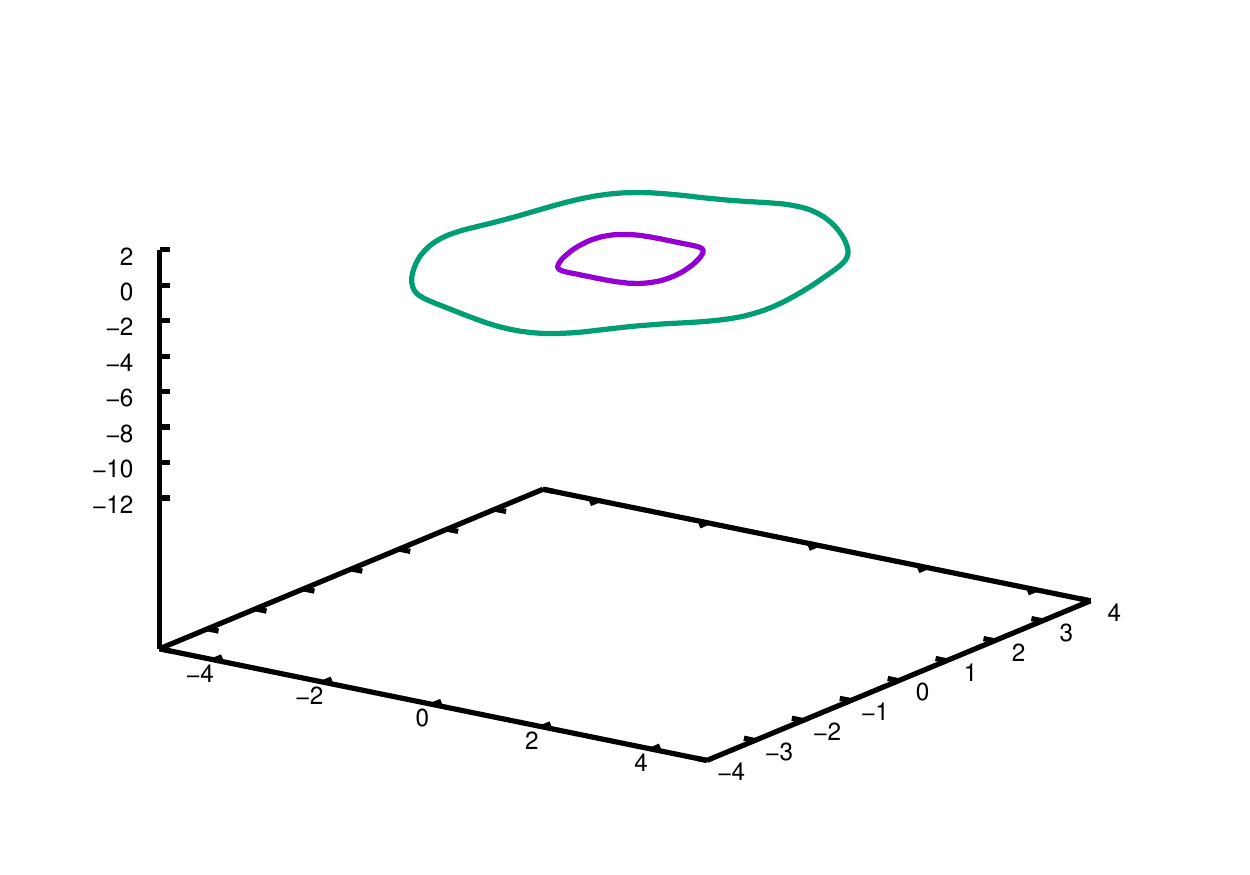}
\includegraphics[width=0.49\textwidth, trim= 1cm 0 1cm  0]{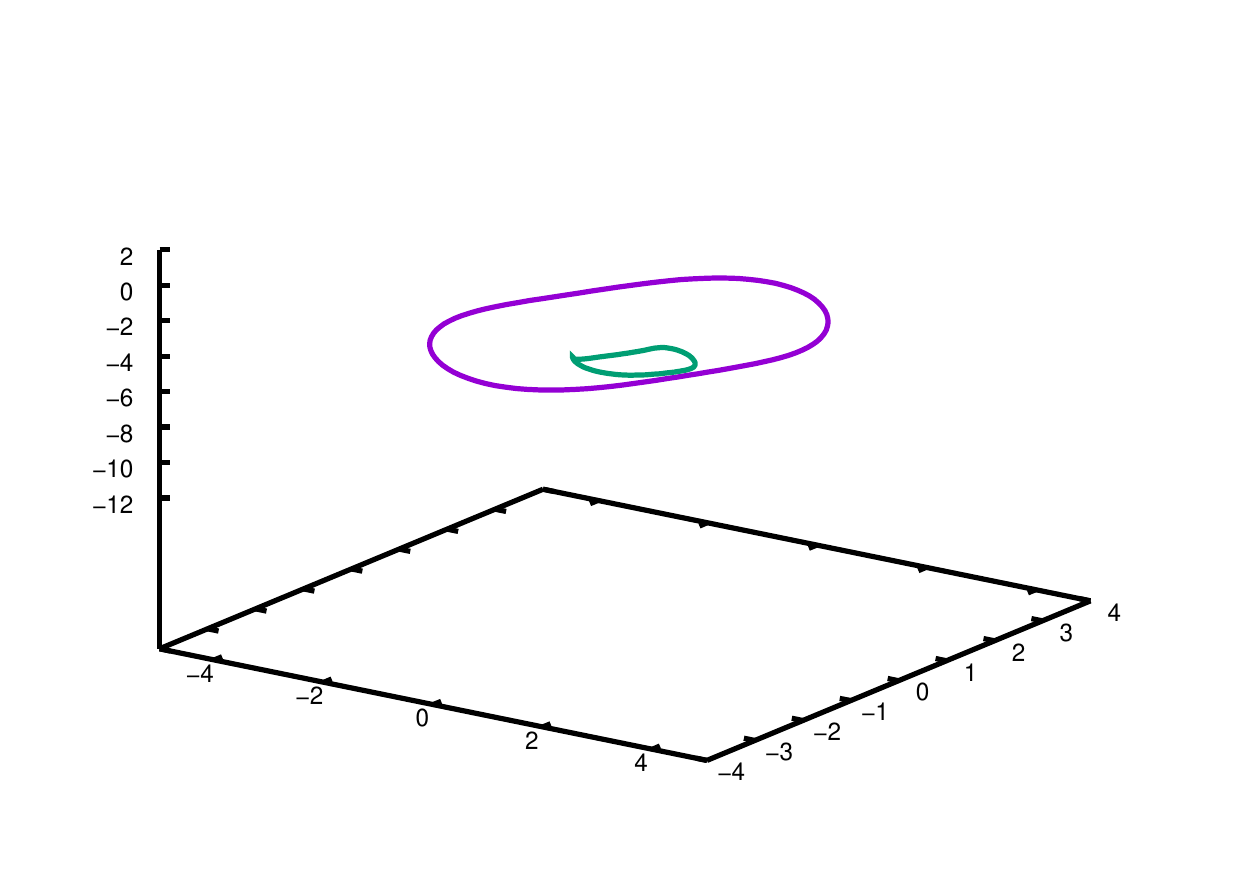}\\
\centerline{\small (a) time $t = 0.0$ \hspace{0.25\textwidth} (b) time $t = 9.0$}
\vspace{-5mm}
\includegraphics[width=0.49\textwidth, trim= 1cm 0 1cm  0]{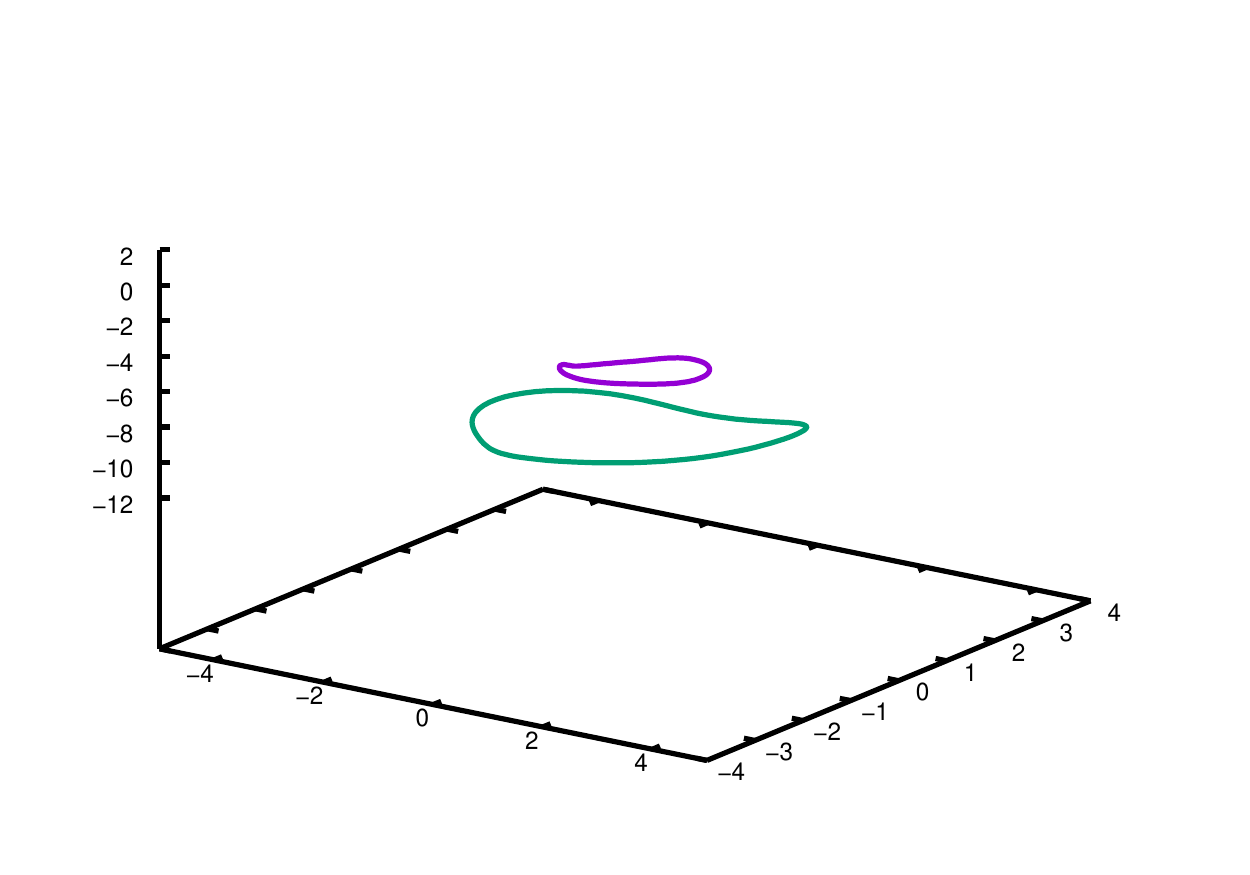}
\includegraphics[width=0.49\textwidth, trim= 1cm 0 1cm  0]{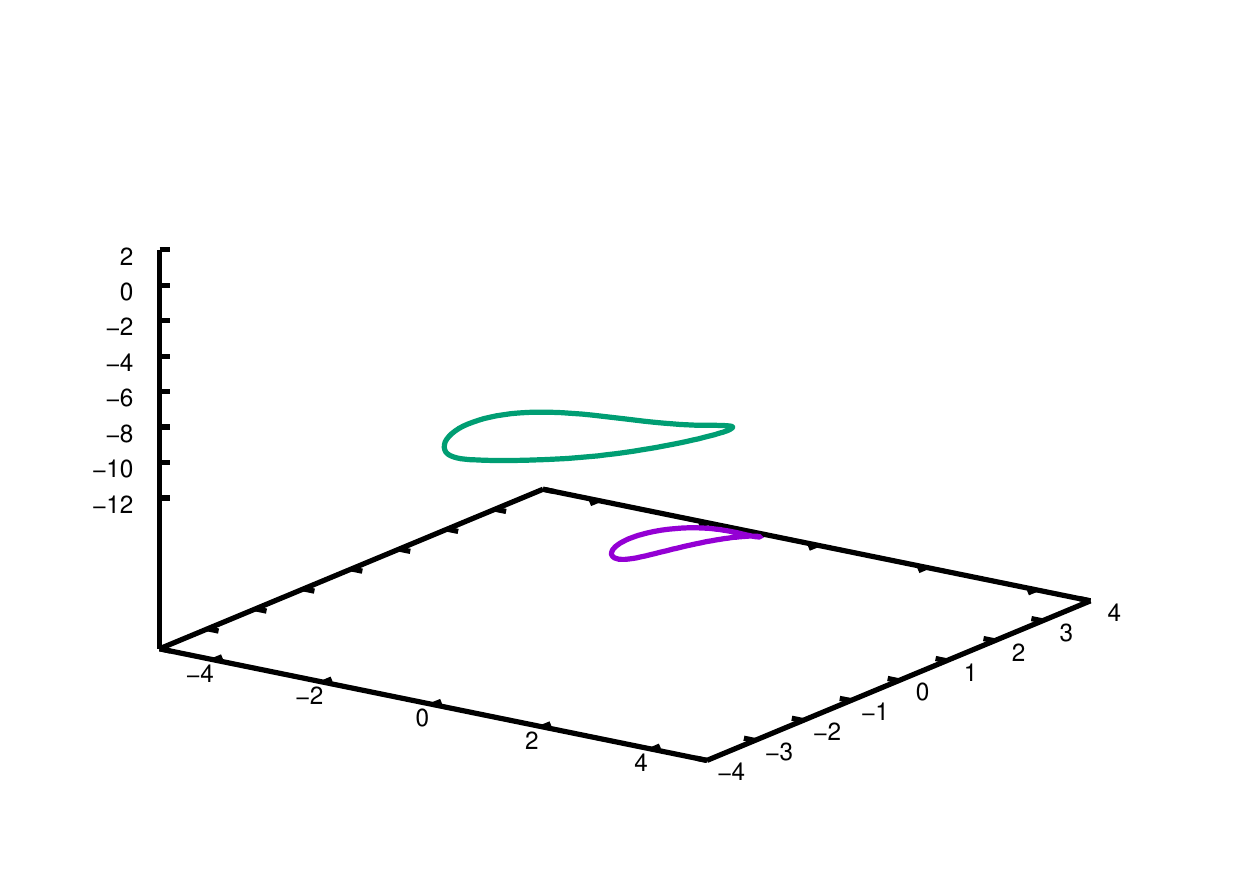}\\
\centerline{\small (c) time $t = 18.0$ \hspace{0.25\textwidth} (d) time $t = 27.0$}
\vspace{-5mm}
\includegraphics[width=0.49\textwidth, trim= 1cm 0 1cm  0]{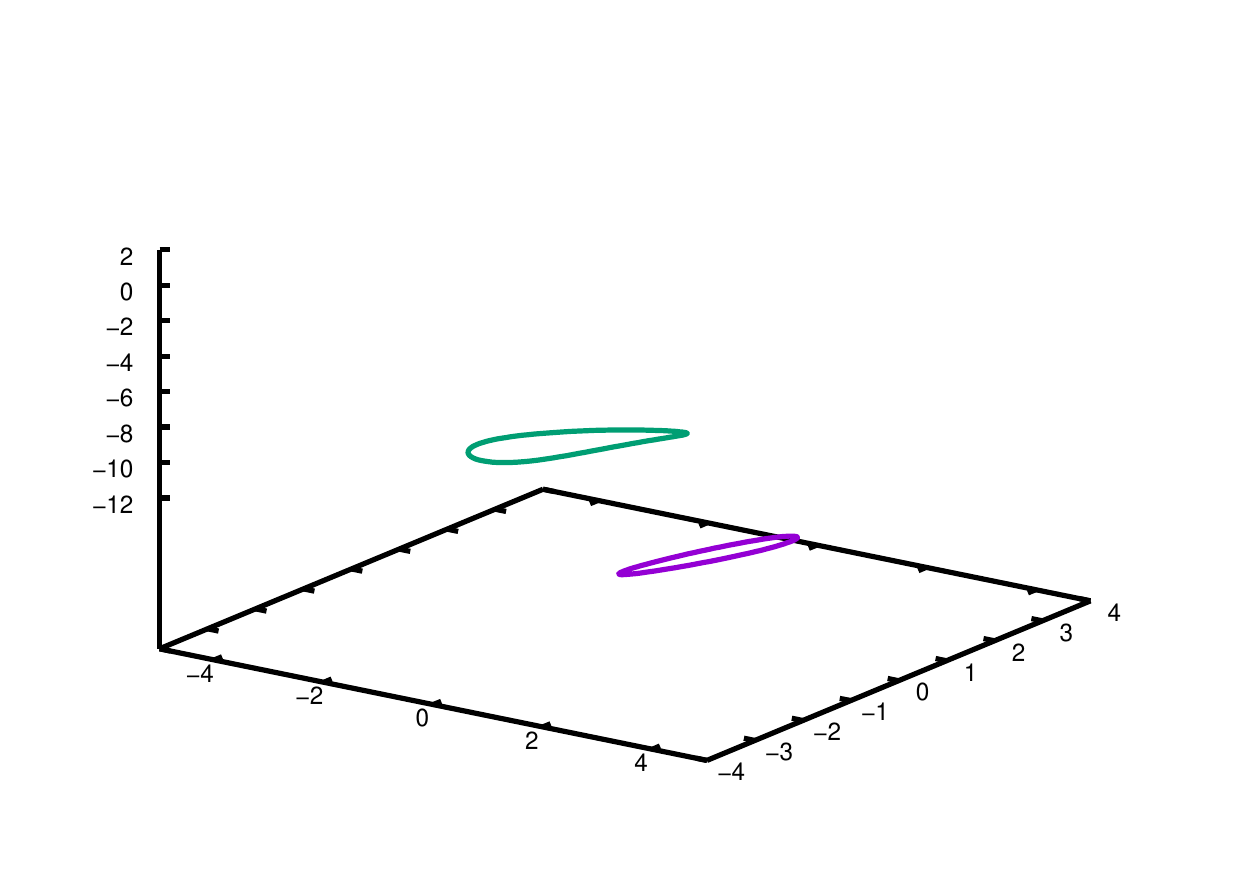}
\includegraphics[width=0.49\textwidth, trim= 1cm 0 1cm  0]{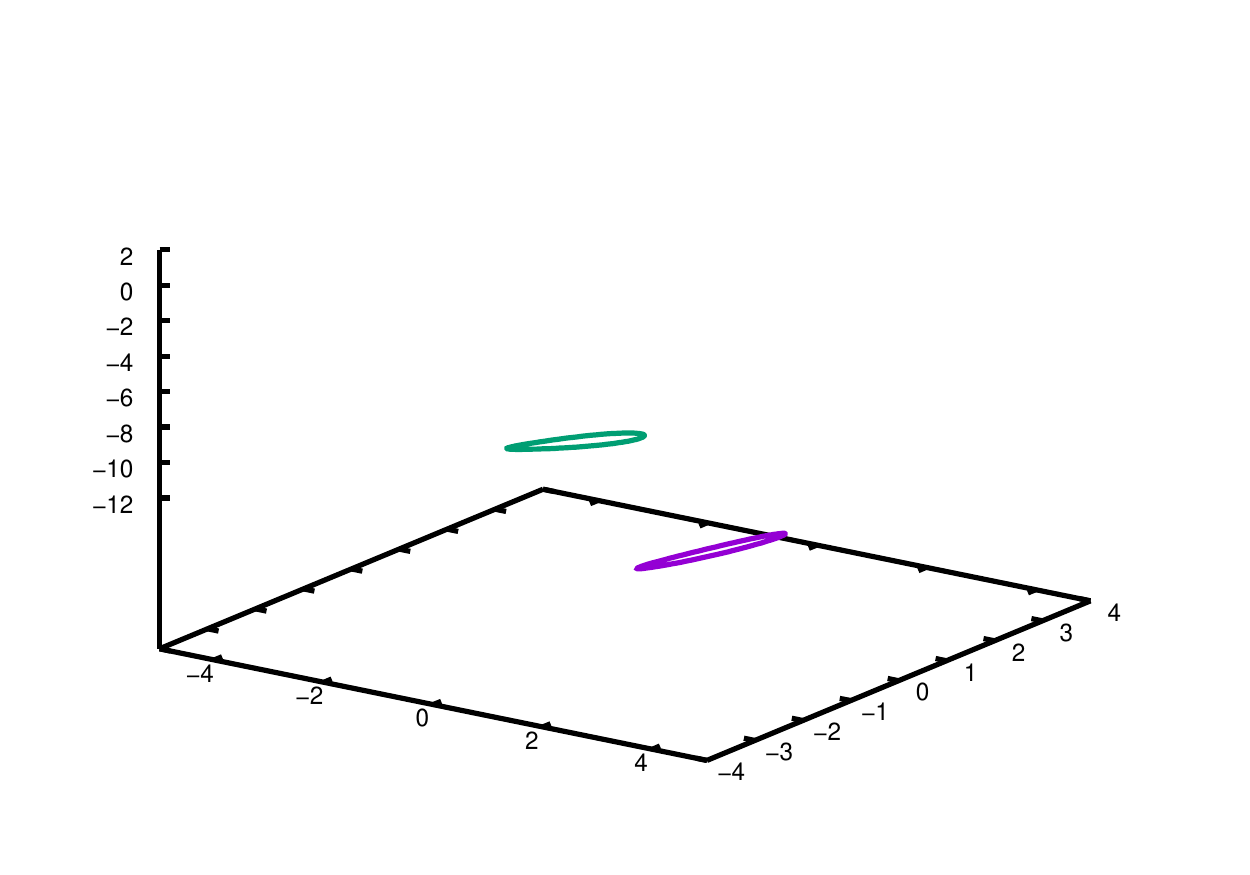}\\
\centerline{\small (e) time $t = 36.0$ \hspace{0.25\textwidth} (f) time $t = 45.0$}
\caption{Example 1: Evolution of space curves with the Biot-Savart type of interaction starting from two vertically perturbed circles showing the "frog leap" dynamics.}

\label{fig:theory1}
\end{center}
\end{figure}

\newpage

\begin{figure}
\begin{center}
\includegraphics[width=0.49\textwidth, trim= 1cm 0 1cm  0]{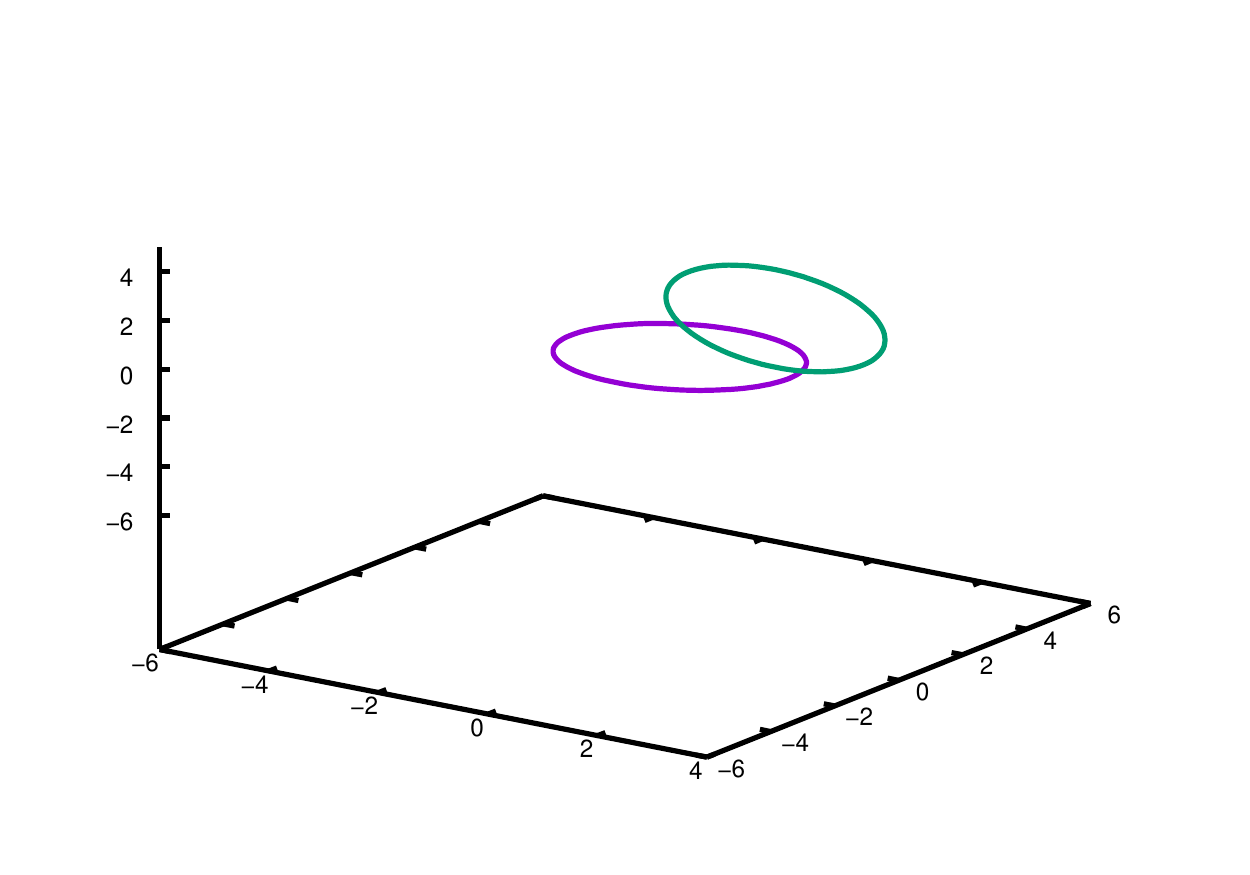}
\includegraphics[width=0.49\textwidth, trim= 1cm 0 1cm  0]{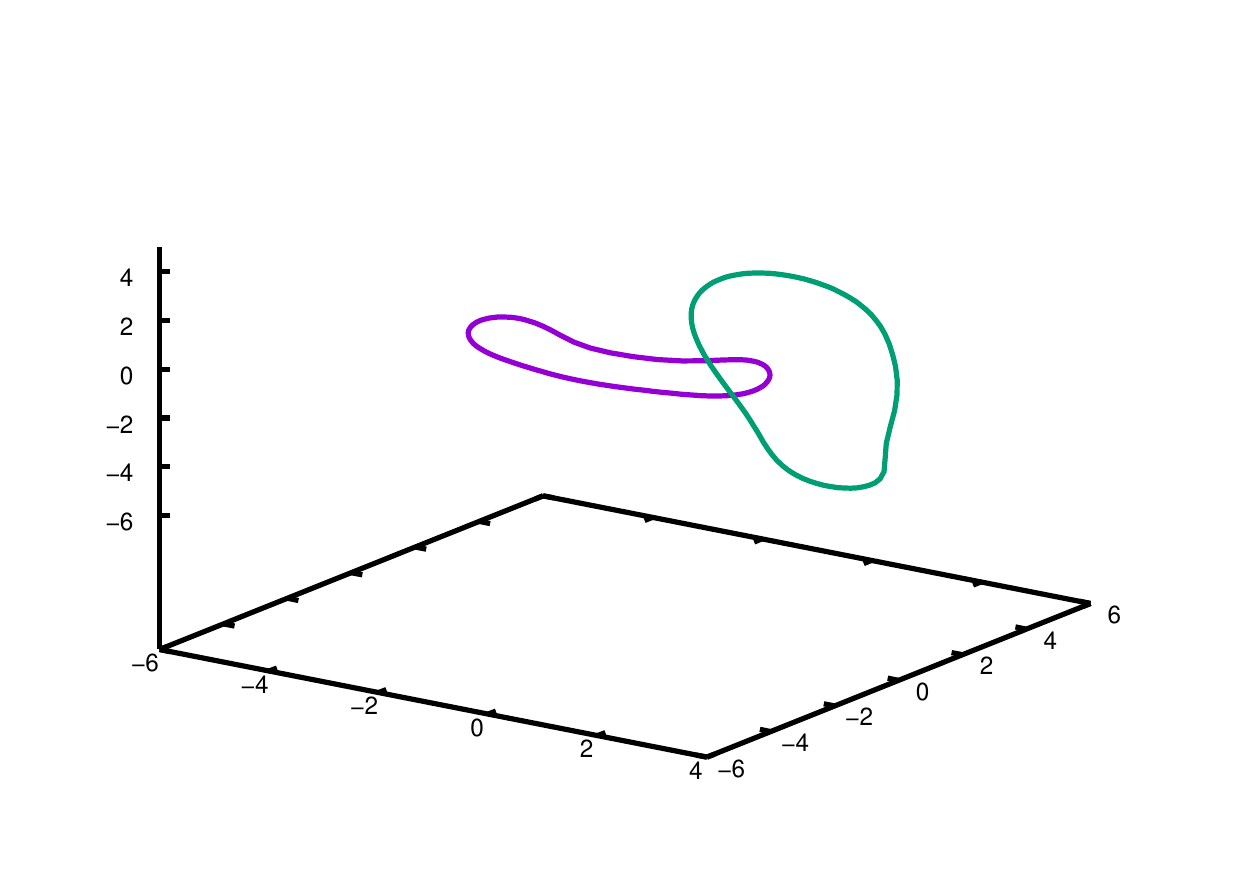}\\
\centerline{\small (a) time $t = 0.0$ \hspace{0.25\textwidth} (b) time $t = 7.2$}
\vspace{-5mm}
\includegraphics[width=0.49\textwidth, trim= 1cm 0 1cm  0]{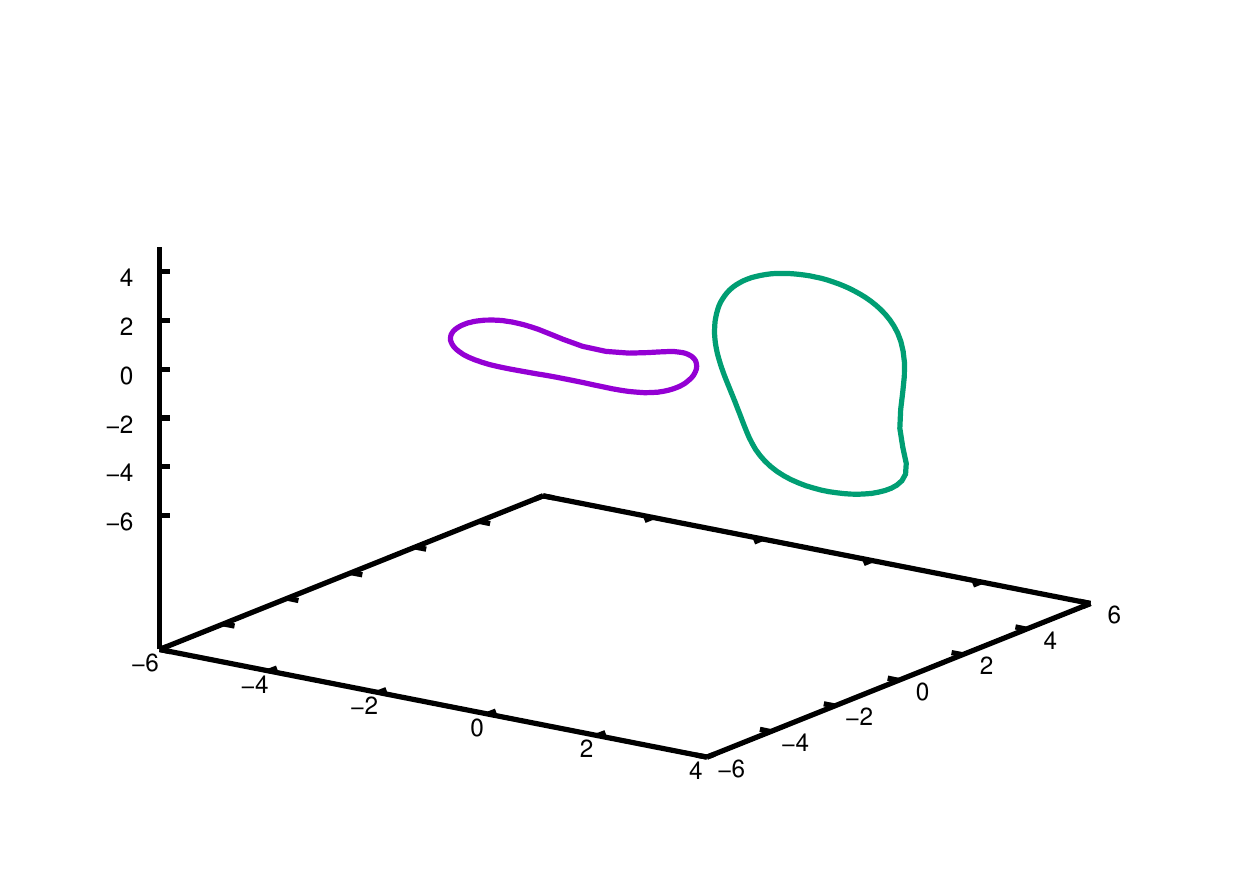}
\includegraphics[width=0.49\textwidth, trim= 1cm 0 1cm  0]{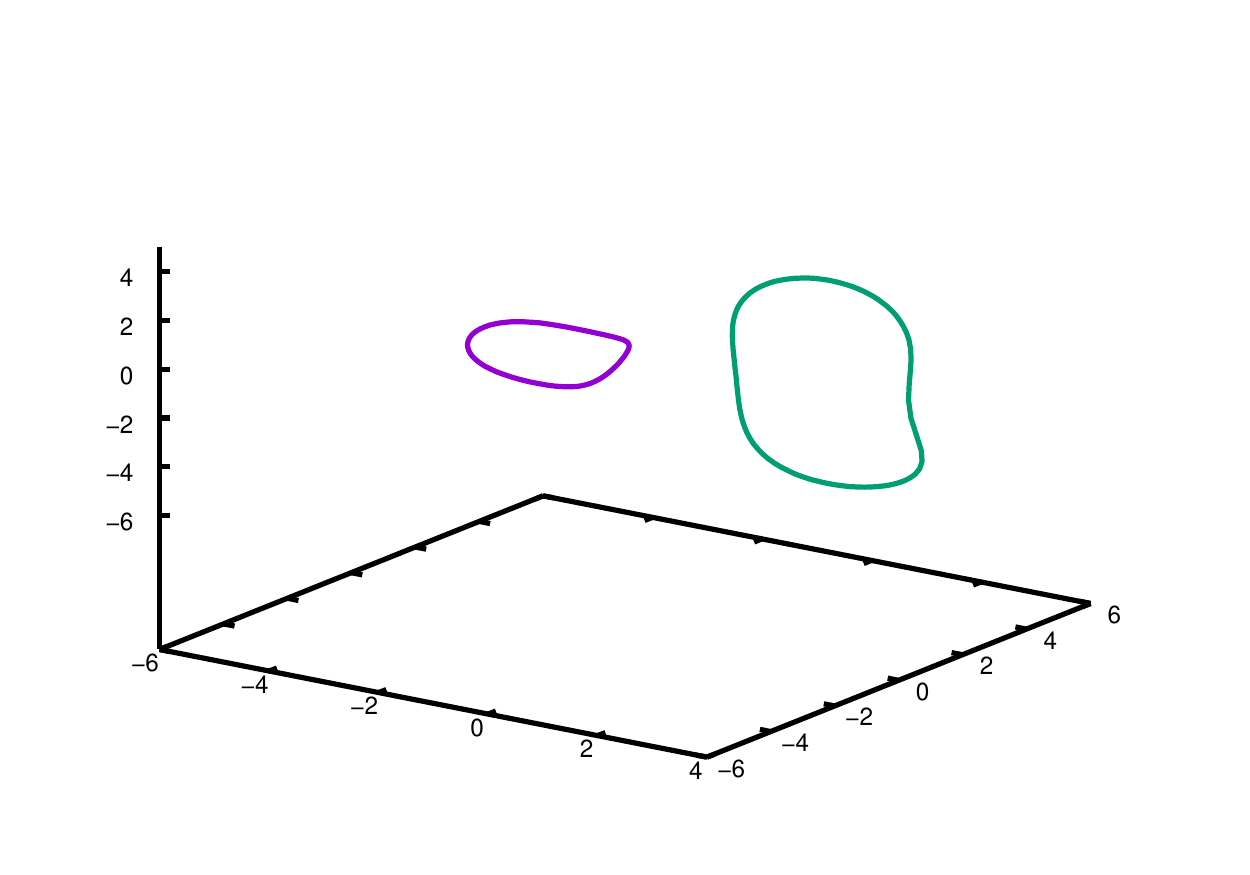}\\
\centerline{\small (c) time $t = 14.4$ \hspace{0.25\textwidth} (d) time $t = 21.6$}
\vspace{-5mm}
\includegraphics[width=0.49\textwidth, trim= 1cm 0 1cm  0]{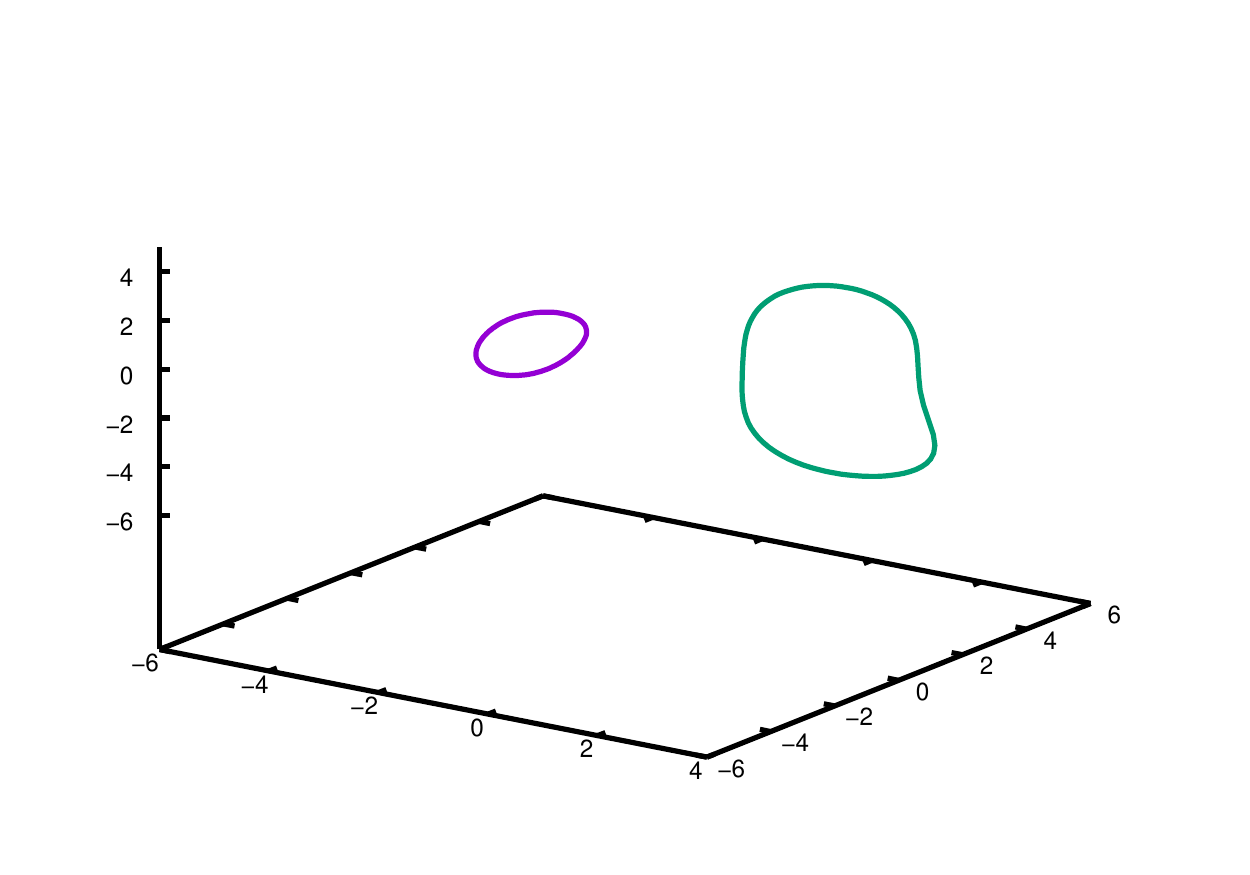}
\includegraphics[width=0.49\textwidth, trim= 1cm 0 1cm  0]{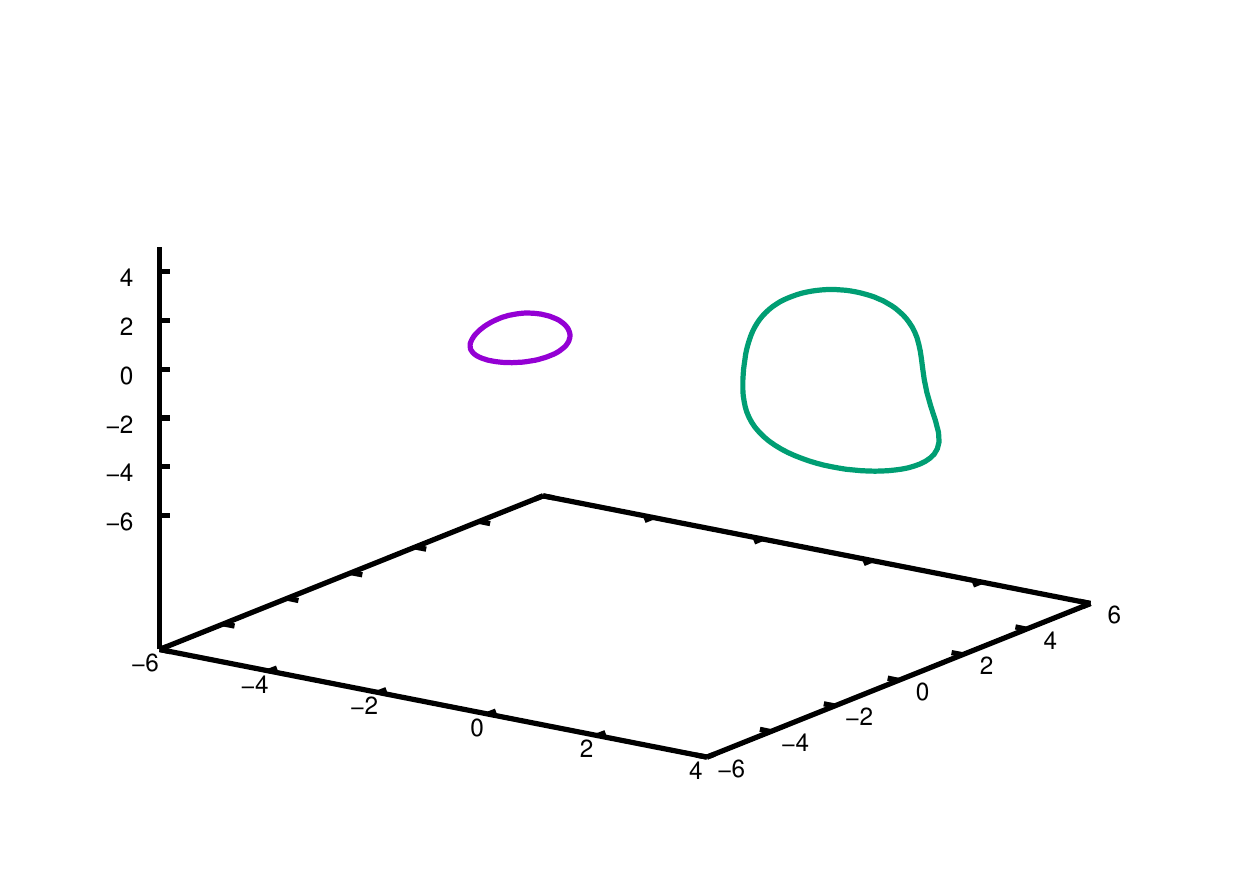}\\
\centerline{\small (e) time $t = 28.8$ \hspace{0.25\textwidth} (f) time $t = 32.0$}
\caption{ Example 2: Evolution of space curves with the Biot-Savart type of interactions starting from two circular curves in perpendicular planes.}
\label{fig:theory2}

\end{center}
\end{figure}

\begin{figure}
\begin{center}
\includegraphics[width=0.49\textwidth, trim= 1cm 0 1cm  0]{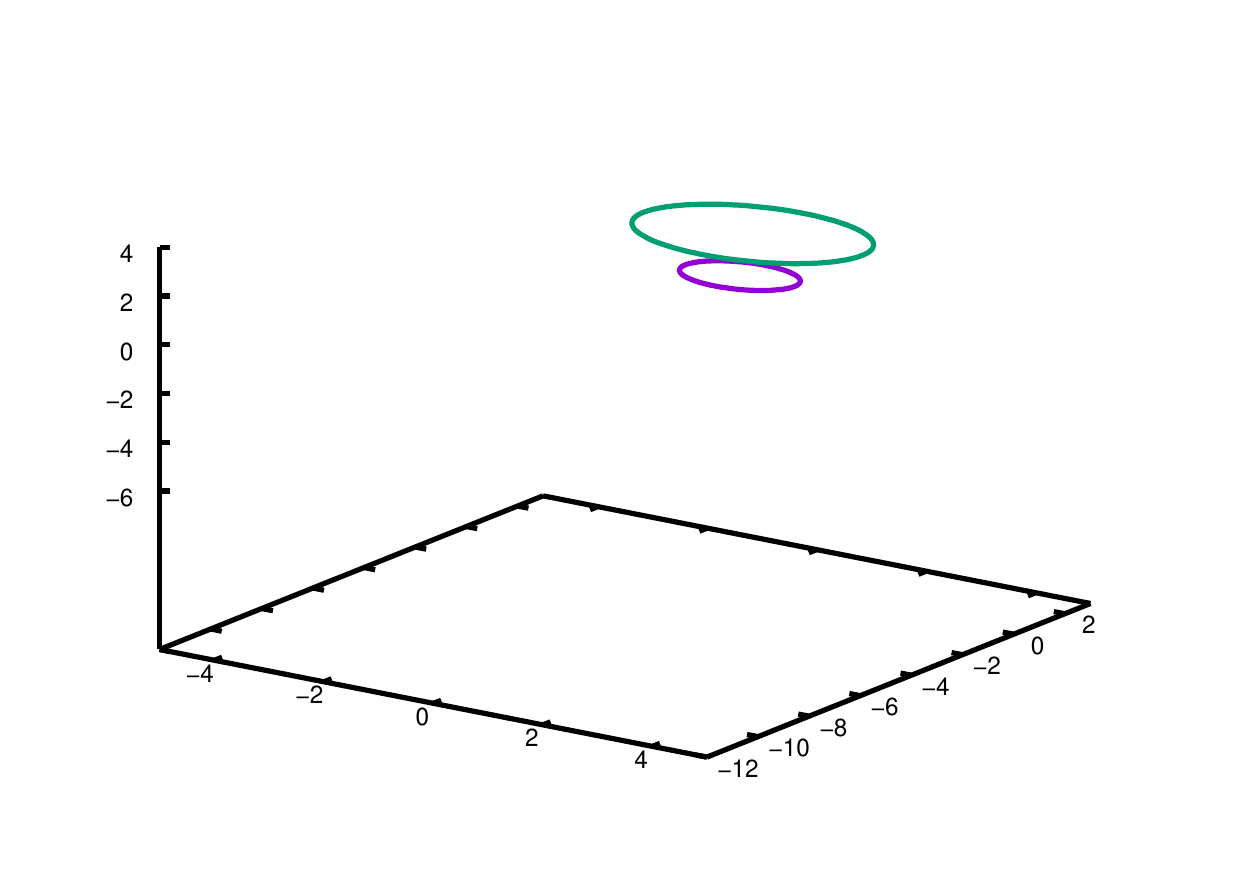}
\includegraphics[width=0.49\textwidth, trim= 1cm 0 1cm  0]{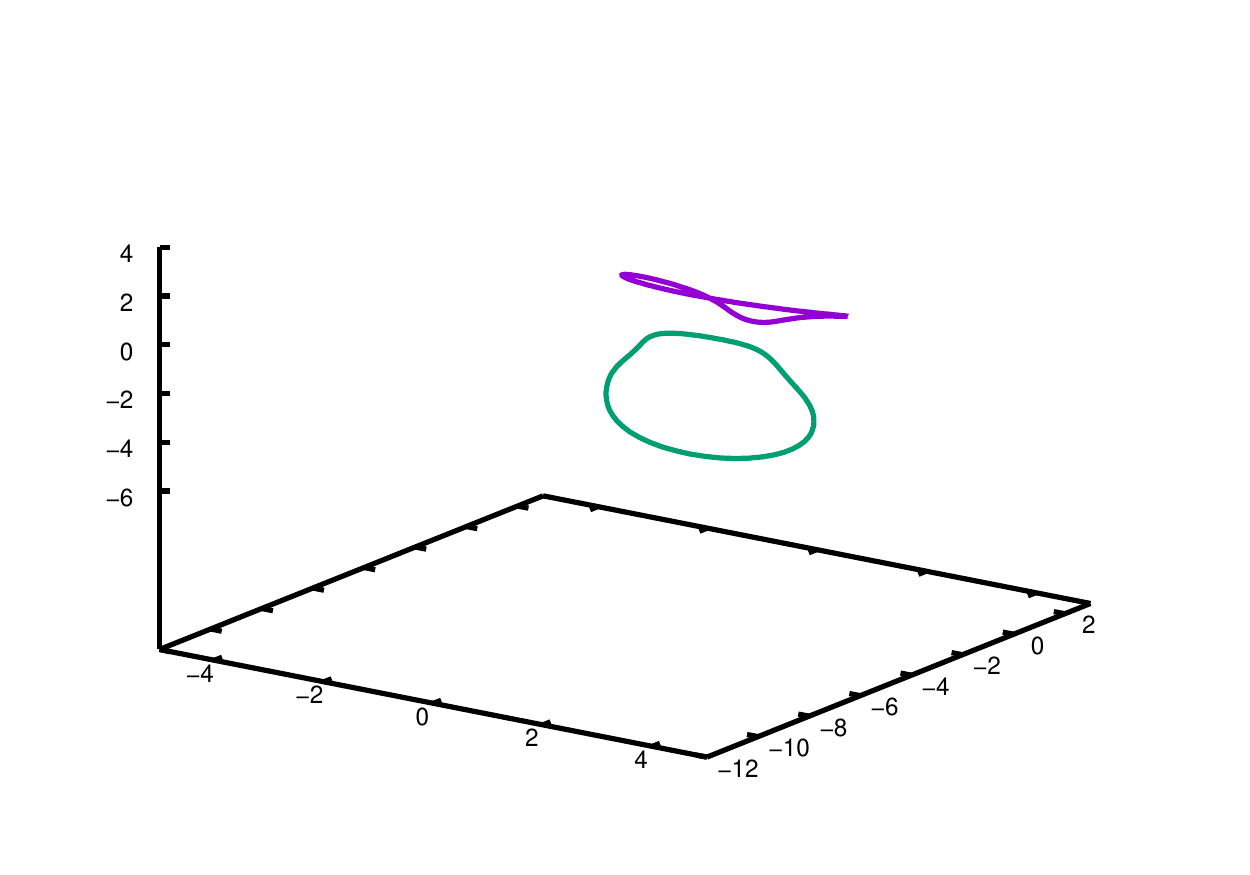}\\
\centerline{\small (a) time $t = 0.0$ \hspace{0.25\textwidth} (b) time $t = 4.4$}
\vspace{-5mm}
\includegraphics[width=0.49\textwidth, trim= 1cm 0 1cm  0]{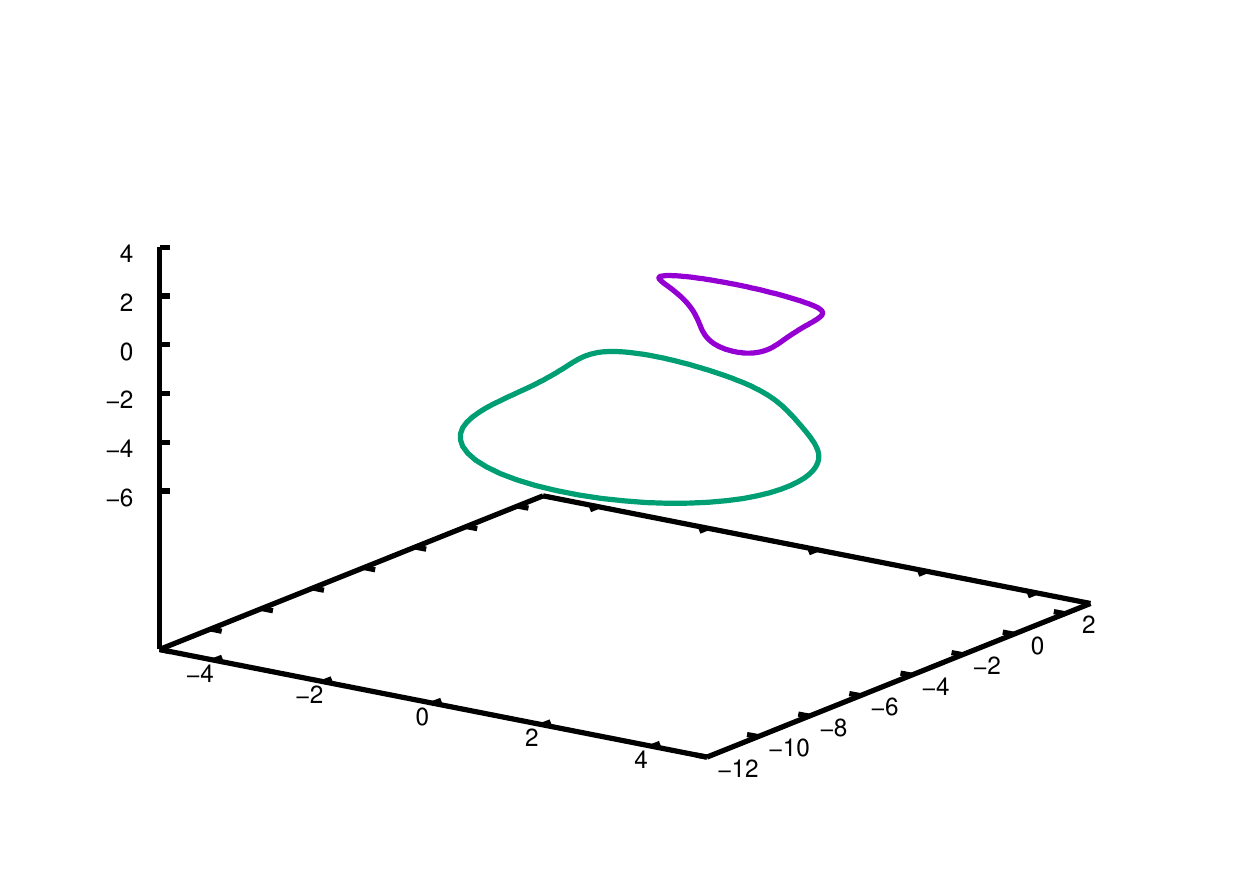}
\includegraphics[width=0.49\textwidth, trim= 1cm 0 1cm  0]{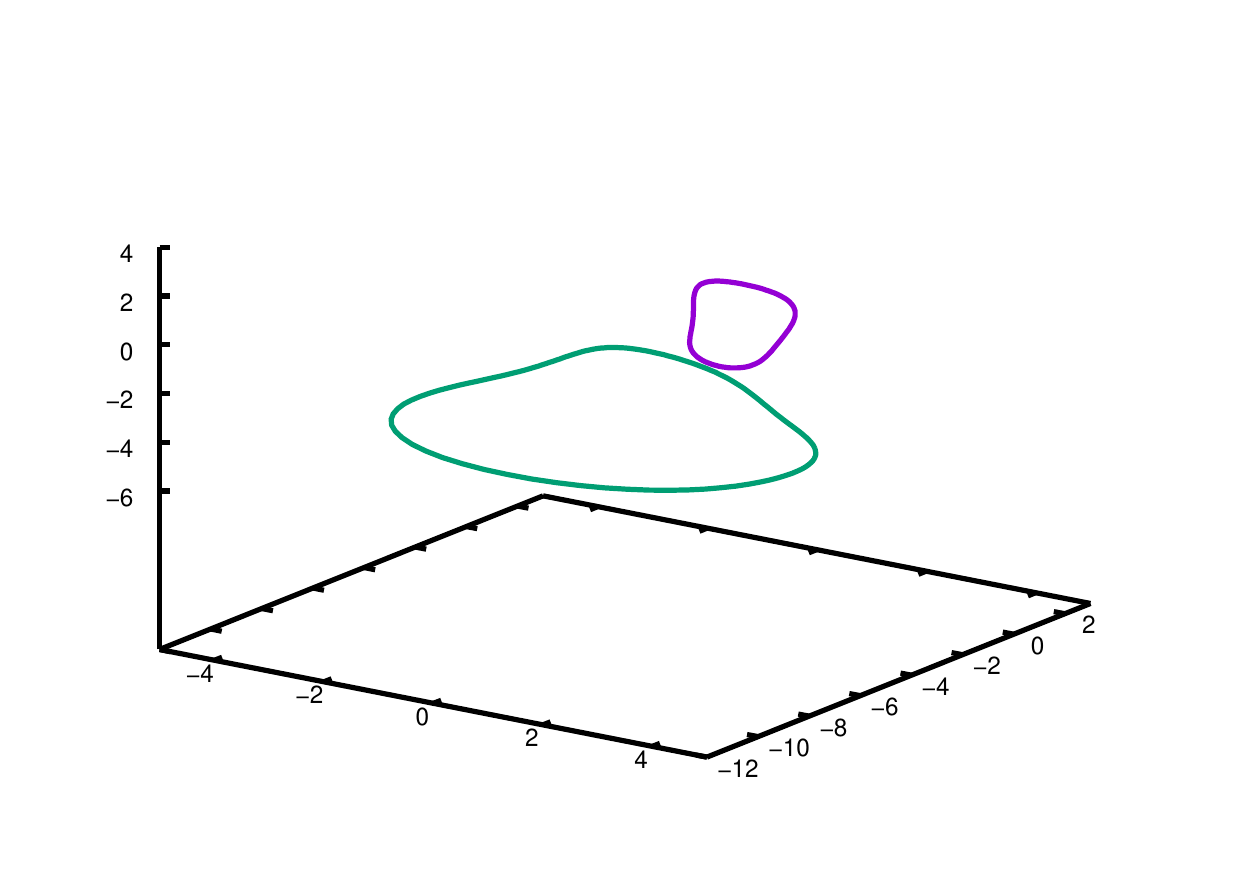}\\
\centerline{\small (c) time $t = 8.8$ \hspace{0.25\textwidth} (d) time $t = 13.2$}
\vspace{-5mm}
\includegraphics[width=0.49\textwidth, trim= 1cm 0 1cm  0]{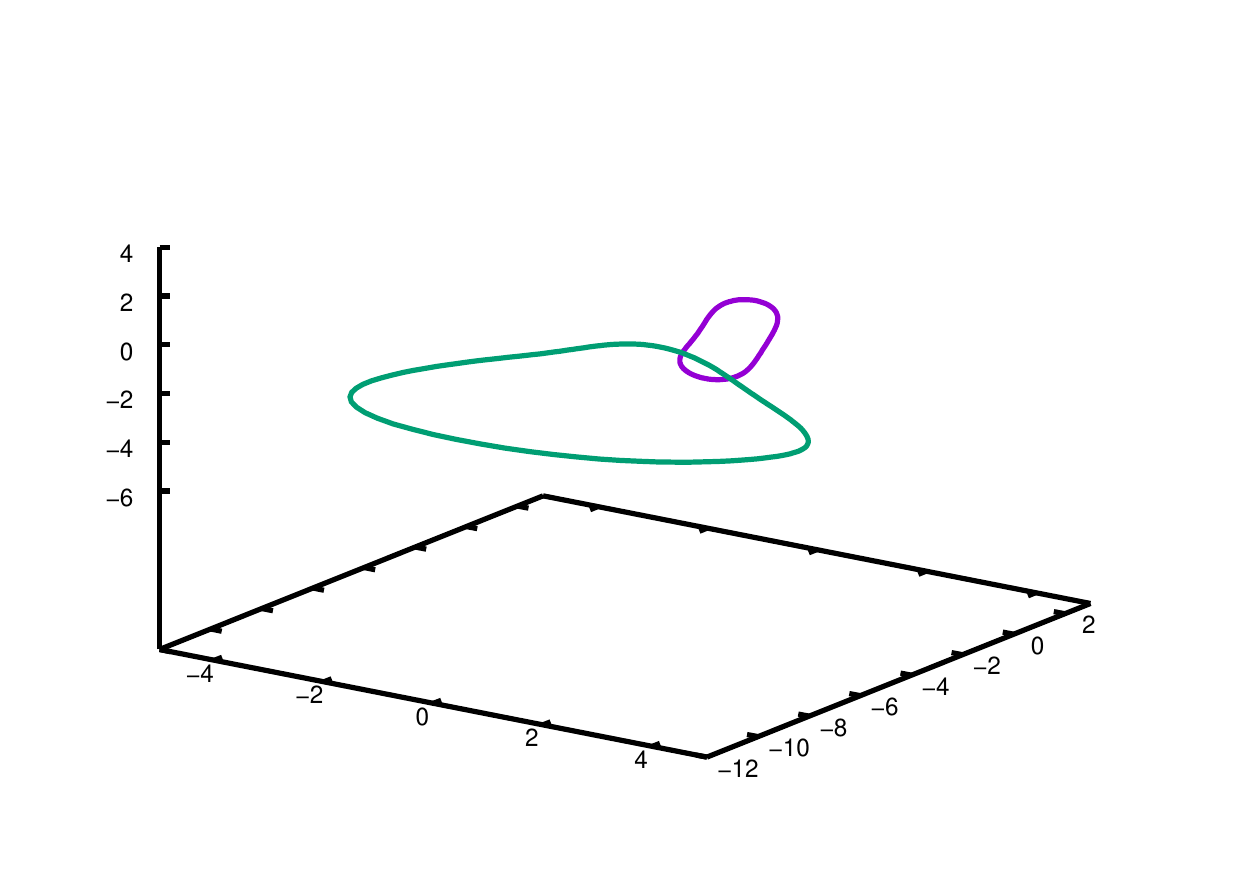}
\includegraphics[width=0.49\textwidth, trim= 1cm 0 1cm  0]{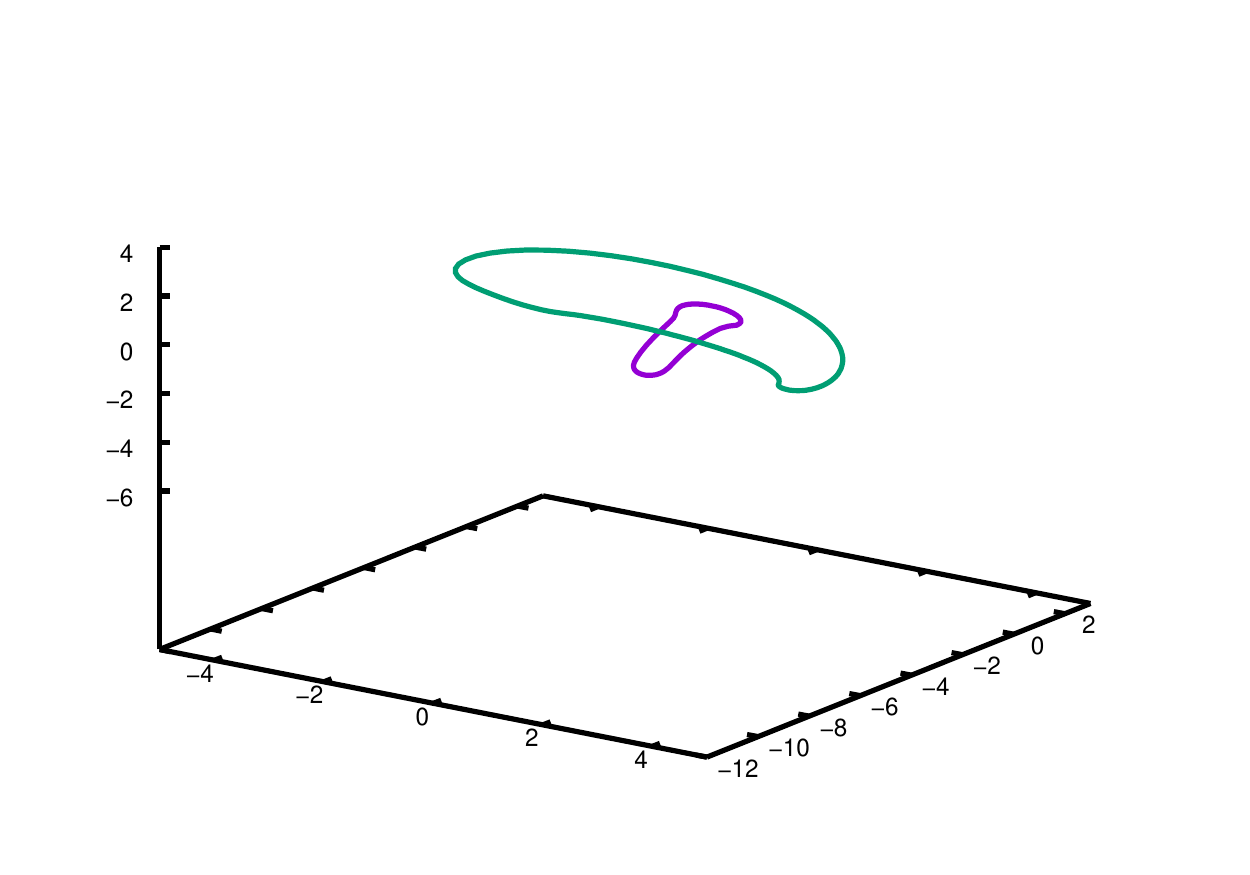}\\
\centerline{\small (e) time $t = 17.6$ \hspace{0.25\textwidth} (f) time $t = 31.6$}
\caption{Example 3: Evolution of space curves with the Biot-Savart type of interactions starting from two non-concentric circular curves showing the "acrobatic" dynamics.}
\label{fig:theory3}

\end{center}
\end{figure}

\subsection{Dynamics of concentric circles under pure binormal flow Biot-Savart type of interactions}\label{binormalonly}

We consider a flow of two vertically concentric circles driven by the system of equations (\ref{biot-savart}). It illustrates the effects of frog leap vortex dynamics (c.f. Mariani and Kontis \cite{Mariani2010}). Parametrizations of vertically concentric circles $\mb{X}^i, i=1,2$ with radii $r_i$ evolving in parallel planes with vertical heights $X_{3i}, i=1,2$, are given by
\[
\mb{X}^i = (r_i \cos 2\pi u, r_i \sin 2\pi u, X_{3i})^T,
\quad
\mb{X}^j = (r_j \cos 2\pi v, r_j \sin 2\pi v, X_{3j})^T,
\quad\text{for} \ u, v\in I.
\]
Then the unit tangent vector $\mb{T}^j = (-\sin 2\pi v, \cos 2\pi v, 0)^T$. In order to compute the integral nonlocal term $\gamma^{ij}(\mb{X}^i)$ is given by means of (\ref{biot-savart-force}) we note that $ds^j = g^j dv = |\partial_v \mb{X}^j| dv = 2\pi r_j$. Furthermore, for $\mb{X}^i = (r_i \cos 2\pi u, r_i \sin 2\pi u, X_{3i})^T$ we have
\begin{equation}
\begin{split}
\partial_s\mb{X}^k\times\partial^2_s \mb{X}^k &= (0,0,1)^T, \  k=i,j,
\\
(\mb{X}^i-\mb{X}^j)\times \mb{T}^j &=
(- z_{ij} \cos 2\pi v, -z_{ij} \sin 2\pi v, r_i \cos 2\pi (v-u) - r_j)^T ,
\\
|\mb{X}^i-\mb{X}^j| &=
|\mb{r}| \sqrt{ 1- \delta \cos 2\pi(v-u) },
\\
\text{where}\ \ z_{ij} = - z_{ji} &= X_{3i} - X_{3j}, \ \ \mb{r} = ( r_1, r_2, z_{12})^T, \ \ \delta= \delta_{ij}=\delta_{ji}=2r_i r_j/|\mb{r}|^2.
\end{split}
\label{ODE-biot-savart}
\end{equation}
Next, we compute the integral over the curve $\Gamma^j$ parametrized by $\mb{X}^j$. The complete elliptic functions of the first kind $K(m)=\int_0^{\pi/2} 1/\sqrt{1- m\sin^2(\vartheta)}d\vartheta$, and the second kind $E(m)=\int_0^{\pi/2} \sqrt{1- m\sin^2(\vartheta)}d\vartheta$ can be used in order to determine all terms entering the integral (\ref{biot-savart-force}) over the curve  $\Gamma^j$. After straightforward calculations employing differentiation of $E$ and $K$ functions, using integration by parts, and relationships between derivatives of $E$ and $K$  one can derive the following explicit expressions for parametric integrals:
\begin{eqnarray*}
I_s(\delta):= \int_0^1 \frac{\sin 2\pi \tilde v}{(1-\delta \sin 2\pi \tilde v)^{3/2}} d\tilde v &=& \frac{2}{\pi} \frac{1}{\delta(1-\delta)\sqrt{1+\delta}} \left(E\left(\frac{2\delta}{1+\delta}\right) - (1-\delta) K\left(\frac{2\delta}{1+\delta}\right) \right) ,
\\
I_c(\delta):= \int_0^1 \frac{\cos 2\pi \tilde v}{(1-\delta \sin 2\pi \tilde v)^{3/2}} d\tilde v &=& 0 ,
\\
I_0(\delta):= \int_0^1 \frac{1}{(1-\delta \sin 2\pi \tilde v)^{3/2}} d\tilde v &=& \frac{2}{\pi} \frac{1}{(1-\delta)\sqrt{1+\delta}} E\left(\frac{2\delta}{1+\delta}\right),
\end{eqnarray*}
for any $|\delta|<1$. Since $\cos 2\pi (v-u) = \sin 2\pi \tilde v$, where $\tilde v = v-u+1/4$ we obtain
\begin{eqnarray*}
\int_0^1 \frac{\sin 2\pi v}{(1-\delta \cos 2\pi (v-u) )^{3/2}} d v &=&
\int_{-u+1/4}^{-u+5/4} \frac{\sin 2\pi (\tilde v + u -\pi/4) }{(1-\delta \sin 2\pi \tilde v )^{3/2}}  d \tilde v
\\
&=& -\int_0^1 \frac{\cos 2\pi (\tilde v + u ) }{(1-\delta \sin 2\pi \tilde v )^{3/2}}  d \tilde v
= I_s(\delta) \sin 2\pi u.
\end{eqnarray*}
Arguing similarly as before, we obtain
\begin{eqnarray*}
\int_0^1 \frac{\cos 2\pi v}{(1-\delta \cos 2\pi (v-u) )^{3/2}} d v &=&
\int_0^1 \frac{\sin 2\pi (\tilde v + u ) }{(1-\delta \sin 2\pi \tilde v )^{3/2}}  d \tilde v
 = I_s(\delta) \cos 2\pi u,
 \\
\int_0^1 \frac{1}{(1-\delta \cos 2\pi (v-u) )^{3/2}} d v &=&
\int_0^1 \frac{1}{(1-\delta \sin 2\pi \tilde v )^{3/2}}  d \tilde v
 = I_0(\delta),
 \\
 \int_0^1 \frac{\cos 2\pi (v-u)}{(1-\delta \cos 2\pi (v-u) )^{3/2}} d v &=&
\int_0^1 \frac{\sin 2\pi \tilde v  }{(1-\delta \sin 2\pi \tilde v )^{3/2}}  d \tilde v
 = I_s(\delta) ,
\end{eqnarray*}

\begin{figure}
    \centering
    \includegraphics[width=0.34\textwidth]{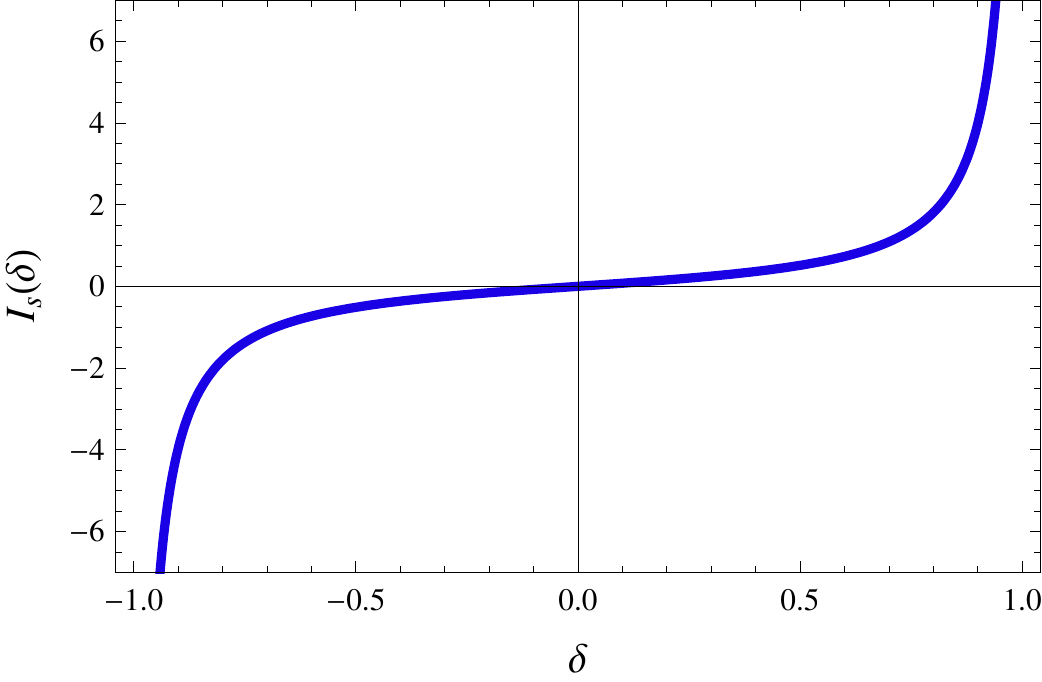}
    \qquad
    \includegraphics[width=0.33\textwidth]{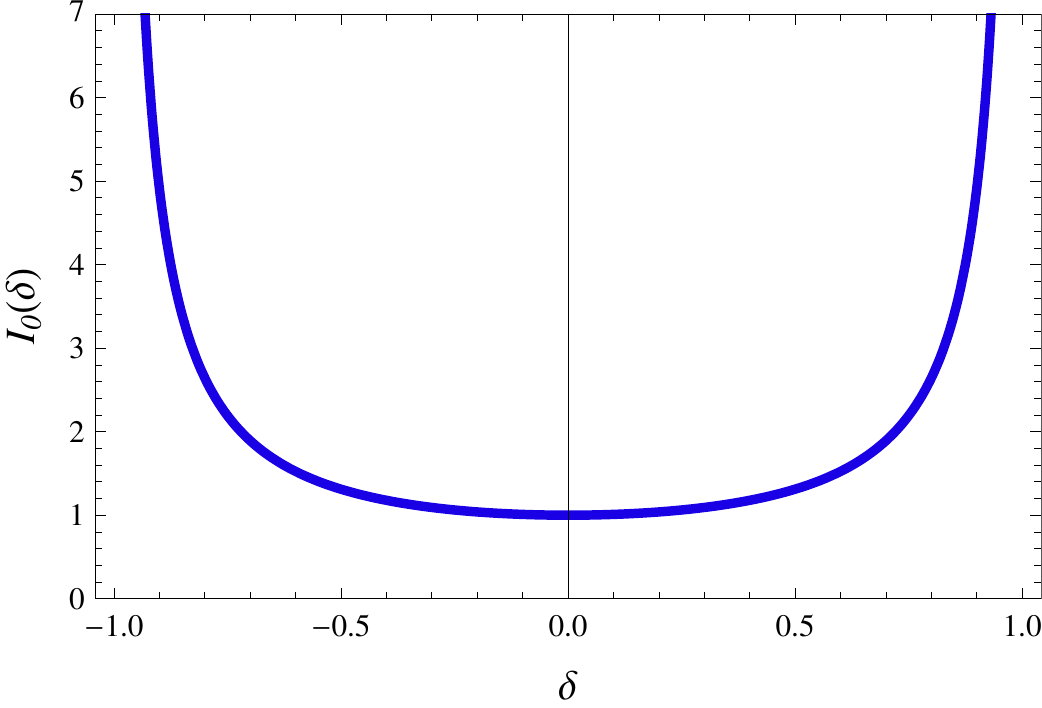}

    \caption{
    Graphs of the functions $I_s(\delta)$ (left) and $I_0(\delta)$ (right).}
    \label{fig:IsI0}
\end{figure}

In summary, we conclude
\begin{eqnarray*}
\gamma^{ij}(\mb{X}^i)
= \frac{2\pi r_j}{|\mb{r}|^3}
\left(
- z_{ij} I_s(\delta) \cos 2\pi u, \
- z_{ij} I_s(\delta) \sin 2\pi u, \
r_i I_s(\delta)  - r_j I_0(\delta)
\right)^T .
\end{eqnarray*}
The radii $r_1, r_2$ and the difference $z_{12}= -z_{21} = X_{31}- X_{32}$ of the heights of  underlying planes satisfy the following  system of nonlinear ODEs:
\begin{eqnarray}
\label{ODE}
\frac{d r_1 }{dt} &=& - \frac{2\pi r_2 z_{12}}{|\mb{r}|^3} I_s(\delta),
\nonumber
\\
\frac{d r_2 }{dt} &=&  \ \ \ \frac{2\pi r_1 z_{12}}{|\mb{r}|^3} I_s(\delta),
\\
\frac{d z_{12} }{dt} &=& \ \ \ \frac{2\pi (r_1^2 - r_2^2)}{|\mb{r}|^3} I_0(\delta), \quad \delta=2r_1 r_2/|\mb{r}|^2, \quad |\mb{r}| = \sqrt{r_1^2+r_2^2+z_{12}^2}.
\nonumber
\end{eqnarray}

If we sum the first equation in (\ref{ODE-biot-savart}) multiplied by $r_1$ with  the second equation multiplied by $r_2$ we conclude that
\[
\frac{d}{dt} (r_1^2(t) + r_2^2(t)) = 0.
\]
Hence the sum of enclosed areas $A(\Gamma^1) + A(\Gamma^2)$ is constant with respect to time, i.e.
\[
A(\Gamma^1_t) + A(\Gamma^2_t) = A(\Gamma^1_0) + A(\Gamma^2_0) \ \ \text{for all}\ t\ge 0.
\]
Therefore the system (\ref{ODE-biot-savart}) has a dynamics of a two-dimensional planar system of ODEs. With regard to the Poincar\'e-Bendixon theorem the $\omega$-limit sets of such a dynamical system consist either of a single fixed point, or a periodic orbit, or it is a connected set composed of a finite number of fixed points together with homoclinic and heteroclinic orbits connecting these fixed points.
In Figure~\ref{fig:planar} we show the solution $(r_1(t), r_2(t), z_{12}(t))$ of the system of ODEs (\ref{ODE-biot-savart}) with initial conditions $r_1(0)=2, r_2(0)=1, z_{12}(0)=3$. The radii of circles are periodically oscillating exchanging their maximums and minimums. Furthermore the difference between moving underlying planes is also oscillating so the one shrinking and expanding circle jumps up and down with respect to the other one.

\begin{figure}
    \centering
    \includegraphics[width=0.4\textwidth]{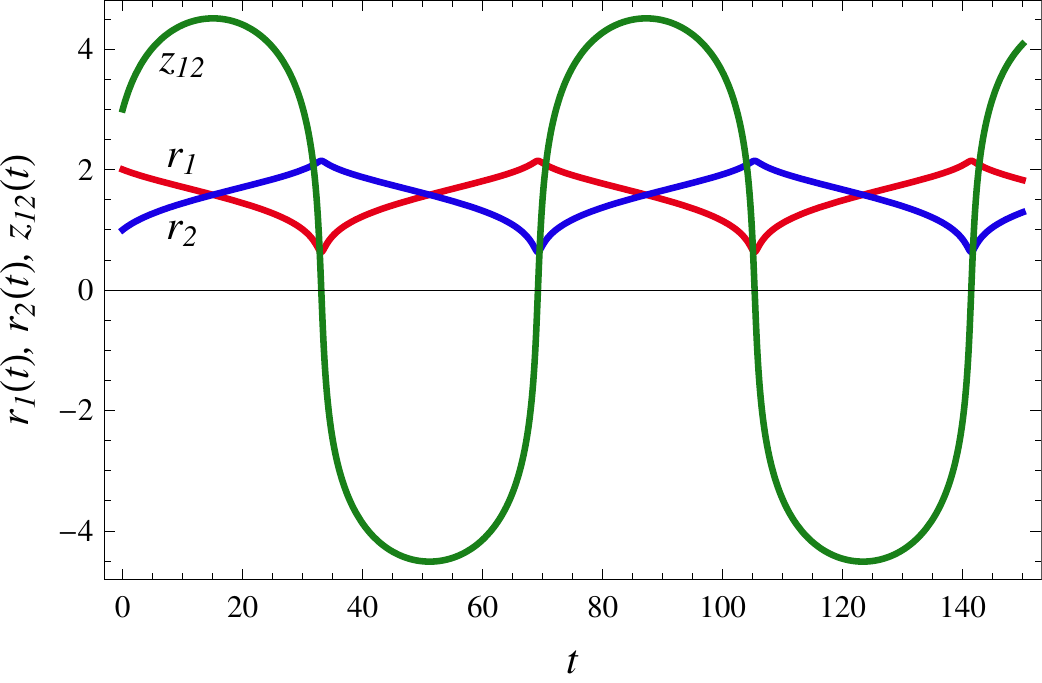}
    \caption{
    Graphs of the functions  $r_1(t)$ (red), $r_2(t)$ (blue), $z_{12}(t)$ (green) solving the nonlinear system of ODEs (\ref{ODE-biot-savart}).}
    \label{fig:planar}
\end{figure}

In general, the evolution of $n$ vertically concentric circles with radii $r_i$, and mutual differences $z_{ij}=X_{3i}-X_{3j}$ of their vertical heights $X_{3i}, i=1,\ldots, n$ satisfy the following system of ODEs:
\begin{eqnarray}
\label{ODE-general}
\frac{dr_i}{dt} &=& -2\pi \sum_{j\not=i} \frac{r_j z_{ij} I_s(\delta_{ij})}{(r_i^2+r_j^2+z_{ij}^2)^{3/2}}, \quad i=1,\ldots, n,
\nonumber
\\
\frac{dz_{ij}}{dt} &=& 2\pi \sum_{k\not=i} \frac{r_k r_i I_s(\delta_{ki}) - r_k^2 I_0(\delta_{ki})} {(r_i^2+r_k^2+z_{ik}^2)^{3/2}}
-
2\pi \sum_{l\not=j} \frac{r_l r_j I_s(\delta_{lj}) - r_l^2 I_0(\delta_{lj})} {(r_j^2+r_l^2+z_{jl}^2)^{3/2}}
,
\\
\nonumber
&&  \text{where}\ \ \delta_{ij} = r_i r_j/(r_i^2+r_j^2+z_{ij}^2),\qquad  i,j=1,\ldots, n.
\end{eqnarray}
Multiplying the differential equation for $r_i$ by $r_i$, summing them for $i=1,\ldots,n$, and taking into account that $z_{ji}= -z_{ij}, \delta_{ij}=\delta_{ji}$ we obtain:
\[
\frac{d}{dt}\sum_{i=1}^n r_i^2(t) = 0, \qquad \text{i.e.  }\quad
\sum_{i=1}^n A(\Gamma^i_t) = \sum_{i=1}^n A(\Gamma^i_0),
\]
for all $t\ge 0$. It means that the flow of vertically concentric circles governed by the geometric law (\ref{ODE-general}) preserves the total area enclosed by the evolving curves.
Since $z_{ij} =  z_{ik} + z_{kj}$, the system (\ref{ODE-general}) can be reduced and computed only for $2n-2$ variables $z_{12}, z_{23}, \ldots, z_{n-1, n}$, and $r_1, \ldots, r_{n-1}$.

\begin{figure}
    \centering
    \includegraphics[width=0.4\textwidth]{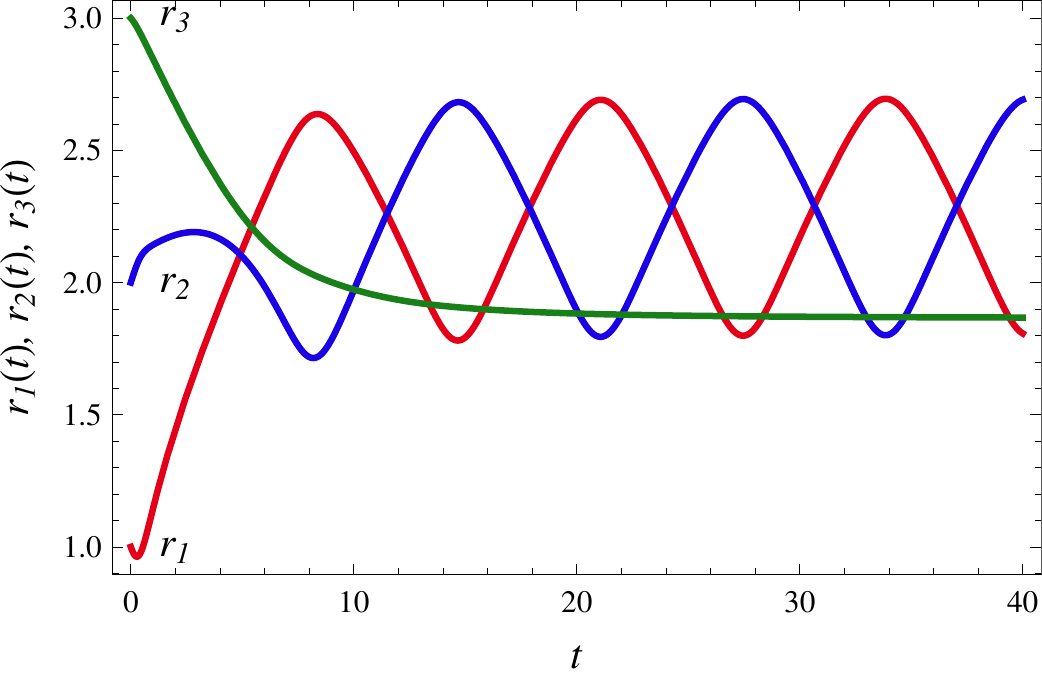}
    \quad
    \includegraphics[width=0.4\textwidth]{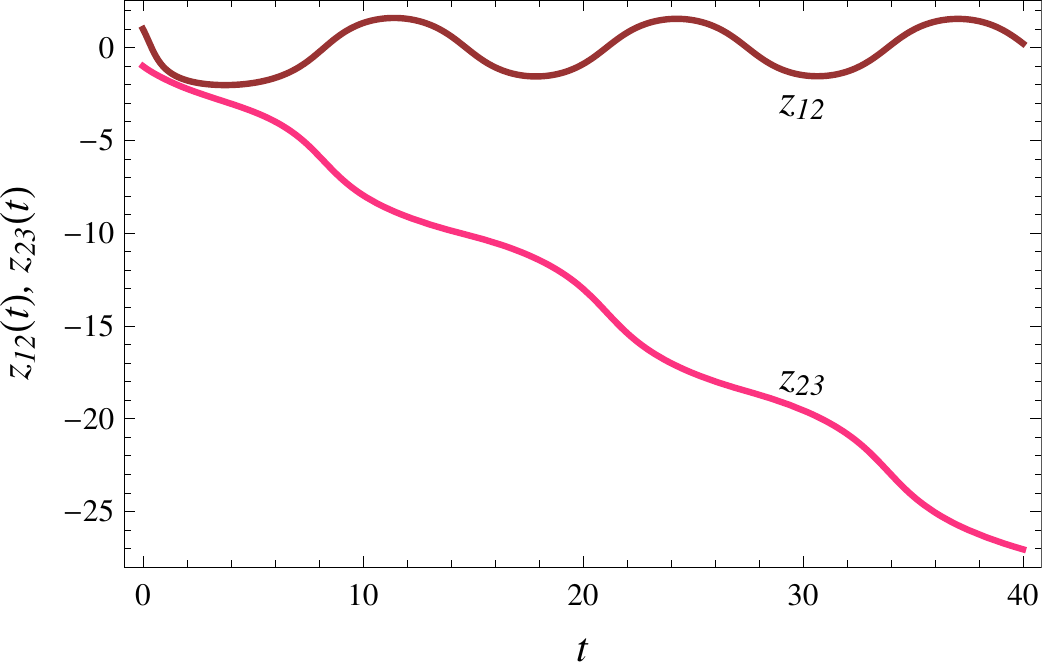}
    \caption{
    Left: Graphs of the functions  $r_1(t)$ (red), $r_2(t)$ (blue), $r_3(t)$ (green). Right: Graphs of the functions  $z_{12}(t)$ (brown), $z_{23}(t)$ (pink).}
    \label{fig:3D}
\end{figure}
In Figure~\ref{fig:3D} we show evolution of radii of $n=3$ vertically concentric circles (left) and their mutual vertical differences $z_{12}, z_{23}$ (right). The dynamical behavior is similar to the $n=2$ case shown in Figure~\ref{fig:planar}) as the radius $r_3$ tends to a steady state, i.e. the circle $\Gamma^3$ converges to a stationary position. The circles $\Gamma^1$ and $\Gamma^2$ are periodically shrinking and expanding as $z_{12}$ oscillates around zero. Their mutual distances $|z_{23}|$, and $|z_{13}|=|z_{12}+z_{23}|$ tend to the third circle $\Gamma^3_t$ tends to infinity as $t\to\infty$.

\section{Conclusion}

In this paper we investigated a curvature driven geometric flow of several curves evolving in 3D  with mutual interactions which can exhibit local as well as nonlocal character and entire curve influences evolution of other curves. We proposed a direct Lagrangian approach for solving such a geometric flow of curves. Using the abstract theory of analytic semi-flows in Banach spaces we proved local existence, uniqueness and continuation of H\"older smooth solutions to the governing system of nonlinear parabolic equations for the position vector parametrization of evolving curves. We applied the method of the flowing finite volume method in combination with the method of lines for numerical discretization of governing equations. We presented several computational examples of evolution of interacting curves. Interaction were modeled by means of the Biot-Savart nonlocal law.

\bigskip

\noindent{\bf Acknowledgement.} This work was partly supported by the Ministry of Education, Youth and Sports of the Czech Republic under the OP RDE grant number CZ.02.1.01/0.0 /0.0/16 019/0000753 "Research centre for low-carbon energy technologies". D. \v{S}ev\v{c}ovi\v{c} was supported by  the Slovak Research and Development Agency under the project APVV-20-0311.

\end{document}